\documentclass[bibalpha]{amsart}
\usepackage{amssymb}
\usepackage{pb-diagram}
\usepackage{comment}
\usepackage{tikz}
\usetikzlibrary{fit}
\usetikzlibrary{patterns}
\usetikzlibrary{positioning}

\newtheorem{theorem}{Theorem}[section]
\newtheorem{qst}[theorem]{Question}
\newtheorem{thm}[theorem]{Theorem}
\newtheorem{prop}[theorem]{Proposition}
\newtheorem{lem}[theorem]{Lemma}

\newtheorem{fact}[theorem]{Fact}
\newtheorem{cor}[theorem]{Corollary}

\newtheorem{lemma}[theorem]{Lemma}

\newtheorem*{thmA}{Theorem A}

\newtheorem*{thmB}{Theorem B}
\newtheorem*{thmC}{Theorem C}
\newtheorem*{thmD}{Theorem D}
\newtheorem*{thmE}{Theorem F}
\newtheorem*{thmF}{Theorem E}

\theoremstyle{definition}
\newtheorem{defn}[theorem]{Definition}

\theoremstyle{remark}

\newcommand{\sub}{\subseteq}

\newcommand{\cl}{\operatorname{Cl}}
\newcommand{\bd}{\operatorname{Bd}}

\newcommand{\Int}{\operatorname{Int}}

\newcommand{\cal}[1]{\ensuremath{\mathcal{#1}}}
\newcommand{\Cal}[1]{\ensuremath{\mathcal{#1}}}

\newcommand{\rvec}{\mathbb{R}_{\mathrm{Vec}}}

\newcommand{\Z}{\mathbb{Z}}
\newcommand{\N}{\mathbb{N}}
\newcommand{\Q}{\mathbb{Q}}
\newcommand{\R}{\mathbb{R}}
\newcommand{\DSig}{D_{\Sigma}}

\begin{document}
\title[A tetrachotomy]{A tetrachotomy for expansions of the real ordered additive group}
\thanks{The first author was partially supported by NSF grants DMS-1300402 and DMS-1654725.
The second author was partially supported by the European Research Council under
the European Union’s Seventh Framework Programme (FP7/2007-2013) / ERC Grant agreement
no. 291111/ MODAG. A preprint of this paper was disseminated under the title ``On continuous functions definable in expansions of the ordered real additive group''.}


\author{Philipp Hieronymi}
\address
{Department of Mathematics\\University of Illinois at Urbana-Champaign\\1409 West Green Street\\Urbana, IL 61801}
\email{phierony@illinois.edu}
\urladdr{http://www.math.illinois.edu/\textasciitilde phierony}

\author{Erik Walsberg}
\address{Department of Mathematics\\ University of California, Irvine}
\email{ewalsber@uci.edu}
\urladdr{https://www.math.uci.edu/\textasciitilde ewalsber}

\date{\today}

\maketitle

\begin{abstract}
Let $\Cal R$ be an expansion of the ordered real additive group. When $\Cal R$ is o-minimal, it is known that either $\Cal R$ defines an ordered field isomorphic to $(\R,<,+,\cdot)$ on some open subinterval $I\subseteq \R$, or $\Cal R$ is a reduct of an ordered vector space. We say $\Cal R$ is field-type if it satisfies the former condition. In this paper, we prove a more general result for arbitrary expansions of $(\R,<,+)$. In particular, we show that for expansions that do not define dense $\omega$-orders (we call these type A expansions), an appropriate version of Zilber's principle holds. Among other things we conclude that in a type A expansion that is not field-type, every continuous definable function $[0,1]^m \to \R^n$ is locally affine outside a nowhere dense set. 
\end{abstract}

\section{Introduction}
A classical theme in model theory, dating back to Zilber's trichotomy conjecture \cite{Zilberconj2}, is to analyze whether model-theoretically tame structures that exhibit well-defined non-linear behavior, actually define fields. While Zilber's original conjecture has famously been proven false by Hrushovski \cite{hrushovski1993new}, Peterzil and Starchenko \cite{PS-Tri} were able to show that for o-minimal structures non-linearity yields a definable field. Restricting ourselves to expansions of the real ordered additive group, we produce in this paper a vast generalization of this result and earlier results of Marker, 
Peterzil and Pillay \cite{MPP}.\newline

\noindent
Throughout this paper, fix a first-order expansion $\mathcal{R} = (\mathbb{R},<,+,\ldots)$ of the ordered additive group of real numbers. \textbf{Definable} without modification means ``$\mathcal{R}$-definable, possibly with parameters'', and $I,J,L$ always range over nonempty bounded open subintervals of $\R$. We say $\mathcal{R}$ is \textbf{field-type} if there are definable functions $\oplus,\otimes : I^2 \to I$ such that $(I,<,\oplus,\otimes)$ is an ordered field isomorphic to $(\mathbb{R},<,+,\cdot)$.
It is easy to see that such $\oplus,\otimes$ must be continuous. We let $\rvec$ be the ordered $\R$-vector space $(\R,<,+,(x \mapsto \lambda x)_{\lambda \in \R})$.\newline

\begin{figure}[b]
\begin{tikzpicture}[every fit/.style={inner sep=0pt, outer sep=0pt, draw}]

\begin{scope}[yshift=1.5cm,y=1cm]
\draw (-2,4.5) -- (-1,3.4);
\node[right] at (-1.5,4.2) {dense $\omega$-order};
\node[right] at (-2.2,3.6) {field-};
\node[right] at (-2.2,3.2) {type};

\node[right] at (1,3.6) {no};
\node[right] at (6,3.6) {yes};
\node[right] at (-1.8,2.2) {yes};
\node[right] at (-1.8,-0.2) {no};
\node [fit={(-1,-1.5) (4,1)}] {};
\node [fit={(4,-1.5) (9,1)}] {};
\node[right] at (5.8,2.8) {\textbf{type C}};
\node[right] at (4.4,2.2) {Every continuous function};
\node[right] at (5.6,1.8) {is definable.};
\node[right] at (-0.2,2.8) {\textbf{type A, field-type}};
\node[right] at (-0.8,2.2) {For any $k$, every definable}; 
\node[right] at (-0.6,1.8) {continuous function is $C^k$};
\node[right] at (-0.2,1.4) {almost everywhere.};
\node[right] at (0,0.5) {\textbf{type A, affine}};
\node[right] at (-0.8,-0.1) {Every definable continuous}; 
\node[right] at (-0.6,-0.5) {function is affine almost};
\node[right] at (0.2,-0.9) {everywhere.};
\node[right] at (5.8,0.5) {\textbf{type B}};
\node[right] at (4.2,-0.1) {Every definable $C^2$ function}; 
\node[right] at (5.8,-0.5) {is affine.};

\end{scope}

\begin{scope}[yshift=2.5cm,y=0.8cm]
\node [fit={(-1,0) (4,3)}] {};
\node [fit={(4,0) (9,3)}] {};
\end{scope}

\end{tikzpicture}
\label{figure1}\caption{A tetrachotomy: Defining a dense $\omega$-order separates \emph{tame} and \emph{wild} expansions, being of field-type distinguishes between \emph{linear} and \emph{non-linear} expansions.} 
\end{figure}

\noindent
It follows from \cite{MPP} and Loveys and Peterzil \cite{PL} that when $\Cal R$ is o-minimal, the expansion $\Cal R$ is either field-type or is a reduct of the ordered vector space $\rvec$. Thus for o-minimal expansions of $(\R,<,+)$ a strong dichotomy into linear and field-type structures is already known. In order to prove corresponding results for arbitrary expansions of $(\R,<,+)$, we need to split the collection of all such expansions into two parts, and then prove Zilber-style dichotomy results separating linear and non-behavior for both parts. This results in a tetrachotomy of all expansions visualized in Figure \ref{figure1}. The criterion we use to separate all expansions of $(\R,<,+)$ is whether or not they define a dense $\omega$-order. An \textbf{$\omega$-orderable set} (or short: an $\omega$-order) is a definable set that is either finite or admits a definable ordering with order type $\omega$.
We say such a set is \textbf{dense} if it is dense in some nonempty open subinterval of $\mathbb{R}$. The behavior of definable sets and functions in $\Cal R$ largely depends on whether or not $\Cal R$ admits a dense $\omega$-order. We say that $\Cal R$ is \textbf{type A} if it does not admit a dense $\omega$-order. We say $\Cal R$ is \textbf{type C} if it defines every compact set and \textbf{type B} if it admits a dense $\omega$-order and is not type C.\newline

\noindent Before discussing questions of linearity versus non-linearity, we want to give a brief rationale for dividing expansions into the types A, B and C. First note that a type C expansion is as wild as can be from a model-theoretic standpoint. Indeed, every projective subset of $[0,1]^n$ in the sense of the projective hierarchy from descriptive set theory (see Kechris \cite[37.6]{kechris}), in particular every Borel function on a bounded domain, is definable in a type C structure. Even the question whether all definable sets in a fixed type C expansion are Lebesgue measurable is independent of ZFC. From a combinatorical model-theoretic point of view, type B expansions are not much better: by \cite[Theorem B]{HW-Monadic} every type B expansion defines an isomorphic copy of the standard model $(\Cal P(\N),\N,\in,+1)$ of the monadic second order theory of $(\N,+1)$, where $\Cal P(\N)$ is the power set of $\N$. It follows that a type B expansion cannot satisfy any Shelah-style model-theoretic tameness property such as $\mathrm{NIP}$ or $\mathrm{NTP}_2$ (see e.g. Simon \cite{Simon-Book} for definitions). It is also easy to see that type B expansions do not satisfy any of the classical logical-geometric tameness notions such as o-minimality, local o-minimality, or d-minimality. Thus all the usual model-theoretic and geometric tameness notions in the literature imply type A. The noteworthy exception is decidability of the theory. Examples of type B expansions with decidable theories are $(\R,<,+,x\mapsto \alpha x,\Z)$ where $\alpha \in \R$ is a quadratic irrational (see \cite{H-Twosubgroups}) and $(\R,<,+,C)$, where $C$ is the middle-thirds Cantor set (see Balderrama and Hieronymi \cite{BH-Cantor}).
We describe another interesting example in Section~\ref{automata}.
Type B expansions have received little attention within tame geometry and  model theory, but have appeared in theoretical computer science (Boigelot, Rassart, and Wolper \cite{BRW}) and fractal geometry (Charlier, Leroy, and Rigo \cite{CLR}). 
One reason might be that all known examples of type B expansions are mutually interpretable with $(\Cal{P}(\N), \N, \in, +1)$.
The theory of the latter structure was shown to be decidable by B\"uchi \cite{Buchi} using automata-theoretic rather then model-theoretic methods.\newline

\noindent Returning to the question of linearity of a given expansion, observe that every type C expansion is field-type. On the other hand, type B expansions are known to not be field-type, and we will derive various results in this paper that clarify the linearity of such structures. Arguably the main contribution of this paper is the following Zilber-style dichotomy theorem for type A expansions.

\subsection{A dichotomy for type A expansions} Before we state the result, we have to introduce some notation. A set $X\subseteq \R^n$ is $\boldsymbol{\DSig}$ if there is a definable family $\{ Y_{r,s} : r,s > 0\}$ of compact subsets of $\R^n$ such that $X = \bigcup_{r,s} Y_{r,s}$, and $Y_{r,s} \subseteq Y_{r',s'}$ if $r\leq r'$ and $s \geq s'$, for all $r,r',s,s'> 0$. A function is $\DSig$ if its graph is $\DSig$, and a definable field $(X,\oplus,\otimes)$ is $\DSig$ whenever $X$ and $\oplus,\otimes$ are $\DSig$.
All open and closed definable sets are $\DSig$ and the collection of $\DSig$-sets is closed under finite intersections, finite unions, cartesian products, and images under continuous definable functions.
A good theory of dimension, definable selection, and generic smoothness: these tools are enough to obtain many results in the o-minimal setting. By Fornasiero et al.~\cite{FHW-Compact} the first two also hold in the type A setting for $\DSig$ sets, and here we obtain the third. This allows us to extend the well-known results in \cite{PS-Tri} from the o-minimal to the more general type A setting.

\begin{thmA}
Suppose $\Cal R$ is type A.
Then the following are equivalent:
\begin{enumerate}
    \item $\Cal R$ is field-type,
    \item there is a $\DSig$ field $(X,\oplus,\otimes)$ with $\dim X > 0$,
    \item there is a $\DSig$ family $(A_x)_{x \in B}$ of subsets of $\R^n$ such that $\dim B \geq 2$, each $A_x$ is one-dimensional, and $A_x \cap A_y$ is zero-dimensional for distinct $x,y \in B$,
    \item there is a definable open $U \subseteq \R^m$ and a $\DSig$ function $f : U \to \R^n$ that is nowhere locally affine.
\end{enumerate}
In particular, if $\Cal R$ is not field-type, then every continuous definable function $U \to \R^n$, where $U$ is a definable open subset of $\R^m$, is locally affine outside a nowhere dense subset of $U$.
\end{thmA}

\noindent The equivalence of (1) and (3) essentially states that a version of Zilber's principle (as stated in \cite[Definition 1.6]{PS-Tri}) holds for type A expansions. 
If $\Cal R$ is o-minimal, then every definable set is $\DSig$. Thus Theorem A indeed generalizes the result from the o-minimal setting.\newline

\noindent 
There are type A expansions that are not field-type, yet define infinite fields. For example, the expansion of $(\R,<,+)$ by all subsets of all $\Z^n$ is locally o-minimal and hence type A (this follows from either Kawakami et al.\cite{KTTT} or Friedman and Miller \cite{FM-Sparse}).
Such pathological examples give an upper limit on what can be proven in the general setting of type A expansions.\newline

\noindent 
As pointed out above, an essential tool we need for the proof of Theorem A is generic $C^k$-smoothness for $\DSig$-functions in type A expansions.
Letting $U$ be a definable open subset of $\R^m$, we say that a property holds \textbf{almost everywhere}, or \textbf{generically}, on $U$ if there is a dense definable open subset of $U$ on which it holds.
It follows from \cite[Theorem D]{FHW-Compact} that if $\Cal R$ is type A, then a nowhere dense definable subset of $\R^k$ has null $k$-dimensional Lebesgue measure. 
So if a property holds almost everywhere in our sense, it holds almost everywhere in the usual measure-theoretic sense.

\begin{thmB}
Suppose that $\Cal R$ is type A.
Fix $k \geq 0$.
Let $U$ be a definable open subset of $\R^m$ and $f : U \to \R^n$ be $\DSig$.
Then $f$ is generically $C^k$ on $U$.
In particular, every continuous definable $f : U \to \R^n$ is generically $C^k$.
\end{thmB}

\noindent 
Thus, if $U \subseteq \R^m$ is open and $f : U \to \R^n$ is continuous and not generically $C^k$ for some $k \ge 1$, then $(\R,<,+,f)$ defines an isomorphic copy of $(\Cal P(\N),\N,\in,+1)$.
Laskowski and Steinhorn~\cite{LasStein} prove Theorem B in the o-minimal setting.
Our proof makes crucial use of their ideas. In particular, we also use a classical theorem of Boas and Widder \cite{BoasWidder}.
Theorem B fails for arbitrary definable functions, as $(\mathbb{R},<,+,\mathbb{Q})$ is type A and the characteristic function of $\mathbb{Q}$ is nowhere continuous. Further observe that Theorem B cannot be strengthened to assert that a continuous function definable in a type A expansion is generically $C^\infty$. Such a result fails already in the o-minimal setting by Rolin, Speissegger, and Wilkie \cite{RSW}.
In the case of expansions of $(\R,<,+,\cdot)$, the one variable case of Theorem B is due to Fornaserio~\cite{F-Hausdorff}.
However, this special case is substantially easier because of definability of division.\newline

\noindent In order to prove Theorem A, we need to combine Theorem B with the following result essentially due to Marker, Peterzil, and Pillay \cite[Section 3]{MPP}. 

\begin{fact}\label{fact:c2}
If $\Cal R$ defines a $C^2$ non-affine function $I \to \R$, then $\Cal R$ is field-type.
\end{fact}

\noindent Their proof is only written to cover the case when $\Cal R$ is o-minimal, but goes through in general (see our proof of Theorem F below). We first prove the one-variable case of Theorem B, apply this to prove Theorem A, and then apply Theorem A to prove the multivariable case of Theorem B. \newline

\noindent We believe that the study of type A expansions is the ultimate generalization of o-minimality in the setting of expansions of $(\R,<,+)$. Introduced by Miller and Speissegger \cite{opencore}, the \textbf{open core} $\Cal R^\circ$ of $\Cal R$ is the expansion of $(\R,<)$ by all open $\Cal R$-definable subsets of all $\R^n$. 
We hope to eventually show that if $\Cal R$ is type A, then $\Cal R^\circ$-definable sets and functions behave in a similar fashion to those definable in o-minimal expansions.
At present we are confined to the collection of $\DSig$ sets. 

\subsection{Linearity of type B expansions} We now discuss the case when $\Cal R$ admits a dense $\omega$-order.
The key result from \cite{discrete} can be restated as follows\footnote{To prove Fact \ref{thm:H}, the reader can either easily redo the proof of \cite[Theorem 1.1]{discrete} or simply apply \cite[Theorem C]{FHW-Compact}.}.
\begin{fact}
\label{thm:H}
Suppose $\Cal R$ expands $(\R,<,+,\cdot)$.
Then there is a dense $\omega$-order if and only if $\Cal R$ defines the set of integers.
\end{fact}

\noindent We immediately obtain the following corollary of Fact~\ref{thm:H}. 

\begin{fact}
\label{thm:H1}
Suppose $\Cal R$ admits a dense $\omega$-order.
Then $\Cal R$ is field-type if and only if $\Cal R$ is type C.
\end{fact}

\noindent Thus studying the linearity of expansions that admit a dense $\omega$-order and are not field-type, reduces to studying the linearity of type B expansions. \newline
 
\noindent  To capture this linearity, we introduce the notion of a weak pole. Recall that a \textbf{pole} is a continuous surjection from a bounded interval to an unbounded interval.
A \textbf{weak pole} is a definable family $\{ h_{d} : d \in E\}$ of continuous maps $h_{d} : [0,d] \to \R$ such that
\begin{itemize}
\item[(i)] $E\subseteq \R_{>0}$ is  closed in $\R_{>0}$ and $(0,\epsilon) \cap E \neq \emptyset$ for all $\epsilon >0$,
\item[(ii)] there is $\delta>0$  such that $[0,\delta] \subseteq h_d([0,d])$ for all $d \in E$.
\end{itemize}
It is easy to see that if $\Cal R$ is field-type, then $\Cal R$ admits a weak pole. Our first result is the following strengthening of Fact \ref{thm:H1}.

\begin{thmC}
\label{thm:weak-pole-main}
Suppose $\Cal R$ admits a dense $\omega$-order.
Then $\Cal R$ is type C if and only if it defines  weak pole.
\end{thmC}

\noindent
Thus a type B expansion cannot define a weak pole. Structures that do not define weak poles, independently of whether they define a dense $\omega$-order, exhibit linear behavior.

\begin{thmD}
\label{thm:weak-pole-conseq}
Suppose $\Cal R$ does not admit a weak pole.
Then 
\begin{enumerate}
    \item If $W \subseteq \R^m$ is open and bounded, then every continuous definable function $f : W \to \R^n$ is bounded.
    \item Every continuous definable function $\R_{>0} \to \R$ is bounded above by an affine function.
    \item Every definable family of linear functions $[0,1] \to \R$ has only finitely many distinct elements.
\end{enumerate}
\end{thmD}

\begin{figure}[t]
\begin{tikzpicture}[every fit/.style={inner sep=0pt, outer sep=0pt, draw}]

\begin{scope}[yshift=1.5cm,y=1cm]
\draw (-2,4.5) -- (-1,3.4);
\node[right] at (-1.5,4.2) {dense $\omega$-order};
\node[right] at (-2.2,3.6) {field-};
\node[right] at (-2.2,3.2) {type};
\draw[]  (-1,-1.5) rectangle (0.5,3.4);
\draw[pattern=north east lines, pattern color=green]  (-1,-1.5) rectangle (0.5,1.0);
\draw[pattern=north east lines, pattern color=green]  
(0.5,-1.5) rectangle (4,0.5);
\draw[pattern=north east lines, pattern color=green]  
(4,-1.5) rectangle (9,1);

\node[right] at (1,3.6) {no};
\node[right] at (6,3.6) {yes};
\node[right] at (-1.8,2.2) {yes};
\node[right] at (-1.8,-0.2) {no};
\node [fit={(-1,-1.5) (4,1)}] {};
\node [fit={(4,-1.5) (9,1)}] {};
\node[right] at (5.8,2) {\textbf{type C}};
\node[right] at (-0.2,2.8) {\textbf{type A, field-type}};
\node[right] at (-0.2,-0.8) {\textbf{type A, affine}};
\node[right] at (5.8,-0.3) {\textbf{type B}};
\node[right, rotate=270] at (0,2) {\textbf{o-minimal}};

\end{scope}

\begin{scope}[yshift=2.5cm,y=0.8cm]
\node [fit={(-1,0) (4,3)}] {};
\node [fit={(4,0) (9,3)}] {};
\end{scope}
\end{tikzpicture}
\caption{The stripped area indicates expansions that do not define weak poles. An o-minimal structure defines a weak pole if and only if it is field-type.
} 
\end{figure}

\noindent Note that Theorem D shows that if $\Cal R$ admits a pole, then $\Cal R$ admits a weak pole. The converse does not always hold. The expansion of $(\R,<,+)$ by all bounded semialgebraic sets clearly admits a weak pole, but does not admit a pole by Pillay, Scowcroft and Steinhorn \cite{PSS}.
It follows from the quantifier elimination for $\rvec$ that $\rvec$ does not admit a weak pole.
Thus if $\Cal R$ is o-minimal, then $\Cal R$ admits a weak pole if and only if $\Cal R$ is field-type. However, there are type A expansions that admit weak poles and are not field-type.
Let $g : 2^{\Z} \times \R \to \R$ be given by $f(t,t') = tt'$.
Then $(\R,<,+,g)$ is a reduct of $(\R,<,+,\cdot,2^\Z)$ and is therefore type A by van den Dries~\cite{vdd-Powers2}.
Delon~\cite{Delon-powers} studied $(\R,<,+,g)$.
It can be deduced from \cite[Theorem 2]{Delon-powers} that $(\R,<,+,g)$ is not field-type.
For $t\in 2^{\Z}$, let $g_t : [0,1] \to \R$ be given by $g_t(t') = tt'$. It is easy to see that $\{ g_t : t \in 2^\Z \}$ is a weak pole. \newline

\noindent 
We also show that continuous definable functions $I \to \R$ in type B expansions satisfy another important property of affine functions.
We say $f: I\to \R$ is \textbf{repetitious} if for every open subinterval $J\subseteq I$
there are $\delta > 0$, $x,y \in J$ such that $\delta < y - x$ and
\[
f(x+\epsilon) - f(x) = f(y + \epsilon) - f(y) \quad \text{for all } 0 \leq \epsilon <  \delta
\]
Affine functions are obviously repetitious. However, a strictly convex function $I \to \R$ is not repetitious. 

\begin{thmF}
Suppose $\Cal R$ is type B.
Then every continuous definable function $I \to \R$ is repetitious.
\end{thmF}

\noindent Thus, if $\Cal R$ defines a continuous function $I \to \R$ that is not repetitious, then $\Cal R$ is field-type.
A special case is that if $\Cal R$ defines a strictly convex function $I \to \R$, then $\Cal R$ is field-type.
While most familiar examples of continuous nowhere locally affine functions are not repetitious, there are continuous repetitious functions that are nowhere locally affine.
If $f =(f_1,f_2): [0,1] \to [0,1]^2$ is the classical Hilbert space-filling curve, then one can check that $f_1$ and $f_2$ are repetitious. \newline

\noindent
It is natural to ask if a type B expansion can interpret an infinite field. By Abu Zaid, Gr\"adel, Kaiser, and Pakusa~\cite{ZGL} the structure $( \Cal{P}(\N), \N, \in, +1)$ does not admit a parameter-free interpretation of an infinite field.
Abu Zaid~\cite{Zaid-thesis} shows that $( \Cal{P}(\N), \N, \in, +1)$ does not interpret $(\R,<,+,\cdot)$, but it appears to be an open question whether $( \Cal{P}(\N), \N, \in, +1)$ interprets an infinite field.

\subsection{Open questions}

\noindent Despite the results above we do not know the full extent to which type B expansions are linear. In particular, we do not know the answer to the following question.
\begin{qst}
\label{qst:type-B}
Suppose that $\Cal R$ is type B. Let $f : I \to \R$ be continuous and definable. Is $f$ generically locally affine?
\end{qst}

\noindent Block Gorman et al. \cite{CRF} give a positive answer to this question for certain natural type B expansions, but the automata-theoretic argument in that paper is unlikely to extend to all type B expansions. By Theorem A this question has a positive answer for type A expansions that are not field-type. Thus a positive answer to Question \ref{qst:type-B} would show that all expansions that are not field-type, satisfy this strong form of linearity. Question~\ref{qst:type-B} can be stated in three non-trivially equivalent forms.

\begin{cor}
Let $f : I \to \R$ be continuous and definable. The following statements are equivalent:
\begin{enumerate}
\item If $\Cal R$ is type B, then $f$ is generically locally affine.
\item If $f$ is nowhere locally affine, then $\Cal R$ is field-type.
\item If $f$ is nowhere $C^k$ for some $k \geq 2$, then $\Cal R$ is type C.
\end{enumerate}
\end{cor}
\begin{proof} Theorem B and Fact~\ref{fact:c2} show that (1) implies (2). Since every nowhere $C^k$ function is also nowhere locally affine, the implication (2) $\Rightarrow$ (3) follows from Theorem B and Fact~\ref{thm:H1}. Since every $C^2$-function definable in a type B expansion is affine, we see that (3) implies (1).
\end{proof}

\noindent Note that statement (2) neither refers to type B or C nor to $\omega$-orders, but rather asks from what kind of objects we can recover a field. We regard this as one of the main open questions in the study of expansions of $(\R,<,+)$.\newline  


\noindent
Another natural question is whether Fact \ref{fact:c2} can be extended to $C^1$ functions. While we are unable to answer this question in general, we give a positive answer under an extremely weak model-theoretic assumption.

\begin{thmE}
Suppose $\mathcal{R}$ does not define an isomorphic copy of the standard model $( \Cal P(\mathbb{N}), \mathbb{N}, \in, +, \cdot)$ of second order arthimetic. If $\Cal R$ defines a non-affine $C^1$ function $f : I \to \mathbb{R}$, then $\Cal R$ is field-type.
\end{thmE}

\noindent Thus every $C^1$ function $I \to \R$ definable in a type B structure with a decidable theory is affine.
This covers the examples of type B structures described above.

\subsection{Applications} We anticipate that applications of the results presented in this paper are numerous. In Section \ref{section:applications} we already collect the most immediate consequences of our work related to descriptive set theory and automata theory. While these results are interesting in their own right, we do not wish to further extend the introduction. We refer the reader to Section \ref{section:applications} for a precise description of these results.

\subsection*{Acknowledgements} We thank Samantha Xu for useful conversations on the topic of this paper, Kobi Peterzil for answering our questions related to \cite{MPP}, and Chris Miller and Michel Rigo for correspondence around Section \ref{section:applications}. We thank the anonymous referee for very helpful comments that improved the presentation of this paper.

\subsection*{Notations}
Let $X \subseteq \R^n$. We denote by $\cl (X)$ the closure of $X$, by $\Int(A)$ the interior of $X$, and by $\bd (X)$ the boundary $\cl(X) \setminus \Int(X)$ of $X$. Whenever $X \subseteq \R^{m+n}$ and $x\in \R^m$, then $X_x$ denotes the set $\{ y \in \R^n \ : \ (x,y) \in X\}$.
A \textbf{box} is a subset of $\R^k$ given as a product of $k$ nonempty open intervals.\newline

\noindent We always use $i,j,k,l,m,n,N$ for natural numbers and $r,s,t,\lambda,\epsilon,\delta$ for real numbers.
We let $ \|x\| := \max \{ |x_1|, \ldots, |x_n| \}$ be the $l_{\infty}$ norm of $x = (x_1,\ldots,x_n) \in \R^n$.

\section{Preliminaries}

\subsection{$D_\Sigma$ sets in type A expansions} 
A set $X\subseteq \R^n$ is $\boldsymbol{\DSig}$ if there is a definable family $\{ Y_{r,s} : r,s > 0\}$ of compact subsets of $\R^n$ such that $X = \bigcup_{r,s} Y_{r,s}$, and $Y_{r,s} \subseteq Y_{r',s'}$ if $r\leq r'$ and $s \geq s'$, for all $r,r',s,s'> 0$. The family $\{ Y_{r,s} : r,s > 0\}$ \textbf{witnessses} that $X$ is $\DSig$.\newline

\noindent 
Every $\DSig$ set is $F_\sigma$, and this might lead us to think of $\DSig$ sets as ``definably $F_\sigma$".
Note however that there can be definable $F_\sigma$ sets that are not $\DSig$.
For example Fact~\ref{thm:sbct} below shows that $\Q$ is not a $\DSig$ set in $(\R,<,+,\Q)$.
A function is $\DSig$ if its graph is $\DSig$.
 We say that a family $\{A_x : x \in \R^m\}$ of subsets of $\R^n$ is $\boldsymbol{\DSig}$ if the set of $(x,y) \in \R^m \times \R^n$ satisfying $y \in A_x$ is $\DSig$.

\begin{fact}[Dolich, Miller and Steinhorn \cite{DMS1}]
\label{fact:basicdsig}
Open and closed definable subsets of $\R^m$ are $D_\Sigma$, finite unions and finite intersections of $D_\Sigma$ sets are $D_\Sigma$, and the image or pre-image of a $D_\Sigma$ set under a continuous definable function is $D_\Sigma$.
In particular, a continuous definable function whose domain is either an open or closed subset of $\R^m$ is $\DSig$.
\end{fact}

\noindent
If $\Cal R$ is o-minimal, then every definable set is a boolean combination of closed definable sets by cell decomposition. So in this situation every definable set is $\DSig$. In general, the complement of a $\DSig$ set is not $\DSig$. In this paper, we will make extensive use of the \textbf{Strong Baire Category Theorem}, or \textbf{SBCT}, established in \cite{FHW-Compact}.

\begin{fact}[SBCT {\cite[Theorem D]{FHW-Compact}}]\label{thm:sbct}
Suppose $\Cal R$ is type A. Then every $\DSig$ subset of $\R^n$ either has interior or is nowhere dense.
\end{fact}

\noindent Another result we use repeatedly is the following \textbf{$\DSig$-selection} result.

\begin{fact}[{\cite[Proposition 5.5]{FHW-Compact}}]\label{thm:selection}
Suppose $\Cal R$ is type A. Let $A \subseteq \R^{m+n}$ be $\DSig$ such that $\pi(A)$ has interior, where $\pi : \R^{m+n} \to \R^m$ is the projection onto the first $m$ coordinates. Then there is a definable open subset $V$ of $\R^m$ such that $V \subseteq \pi(A)$ and a continuous definable $f: V \to \R^n$ such that $(p, f(p)) \in A$ for all $p \in V$.
\end{fact}

\noindent Theorem~\ref{thm:dsig-cont} is an easy consequence of Fact~\ref{thm:selection}.
We leave the details to the reader.

\begin{thm}
\label{thm:dsig-cont}
Suppose $\Cal R$ is type A. Let $U \subseteq \R^m$ be definable open and let $f : U \to \R^n$ be $\DSig$.
Then $f$ is generically continuous.
\end{thm}

\noindent For the continuous functions in type A structures the following weak monotonicity theorem for type A structures.
\begin{fact}[{\cite[Theorem 4.3]{HW-Monadic} and \cite[Fact 3.3]{FHW-Compact}}]\label{fact:weakmonot}
Suppose $\Cal R$ is type A. Let $Z\subseteq \R^n$ be definable and let $(f_z:\mathbb{R}\to \mathbb{R})_{z\in Z}$ be a definable family of continuous functions.  Then there is a definable family $(U_z)_{z\in Z}$ of open dense subsets of $\mathbb{R}$ such that for every $z\in Z$ the function $f_z$ is strictly increasing, strictly decreasing, or constant on each connected component of $U_z$.
\end{fact}

\noindent The uniformity in the statement of the weak monotonicity theorems is not explicit in the literature, but an inspection of the proof of \cite[Theorem 4.3]{HW-Monadic} shows that the definable open set $U$ is clearly constructed uniformly in the parameters defining $f$.\newline

\noindent In a few place throughout this paper we will refer to the dimension of a $\DSig$ set in a type A expansion. It is necessary to explain what dimension we refer to. Given $X \subseteq \R^n$ we let $\dim(X)$ be the topological dimension of $X$.
\textbf{Topological dimension} here refers to either small inductive dimension, large inductive dimension, or Lebesgue covering dimension. These three dimensions coincide on all subsets of $\R^n$ (see Engelking \cite{Engelking} for details and definitions). Model-theorists usually consider as a dimension of a subset $X$ of $\R^n$ the maximal $k$ for which there is a coordinate projection $\rho : \R^n \to \R^k$ such that $\rho(X)$ has nonempty interior. In \cite{FHW-Compact} this is called the \textbf{naive dimension} of $X$. In general, this naive dimension is not well-behaved for arbitrary subsets of $\R^n$ and does not equal the topological dimension. However, for $\DSig$ sets these notions of dimension coincide.

\begin{fact}[{\cite[Proposition 5.7, Theorem F]{FHW-Compact}}]\label{thm:factsaboutdsig}
Suppose $\Cal R$ is type A. Let $X \subseteq \R^n$ be $D_\Sigma$.
Then $\dim(X)$ is equal to the maximal $k$ for which there is a coordinate projection $\rho : \R^n \to \R^k$ such that $\rho(X)$ has nonempty interior. Moreover,
$\dim \cl(X) = \dim(X)$.
\end{fact}

\noindent Corollary~\ref{cor:to-selection} follows from $\DSig$-selection and Fact~\ref{thm:factsaboutdsig}.

\begin{cor}
\label{cor:to-selection}
Suppose $\Cal R$ is type A.
Suppose $X \subseteq \R^n$ is $\DSig$ and $\dim X \geq d$.
Then there is a nonempty definable open $U \subseteq \R^d$ and a continuous definable injection $f : U \to X$.
\end{cor}

\noindent We also make use of the following.

\begin{fact}[{\cite[Theorem F]{FHW-Compact}}]
\label{fact:dim-bijection}
Suppose $\Cal R$ is type A.
Let $X \subseteq \R^n$ be $\DSig$ and suppose $f : X \to \R^m$ is definable and continuous.
Then $\dim f(X) \leq \dim X$ and $\dim f(X) = \dim X$ when $f$ is a bijection.
\end{fact}

\noindent Finally, we also need the following consequence of the Kuratowski-Ulam theorem (see~\cite[Theorem 8.41]{kechris}).

\begin{fact}
\label{fact:KU}
Suppose $U \subseteq \R^m, V \subseteq \R^n$ are nonempty and open.
Let $A$ be an $F_\sigma$ subset of $U \times V$.
If $A_x$ contains a dense open subset of $V$ for all $x \in U$, then $A$ contains a dense open subset of $U \times V$.
\end{fact}

\subsection{Expansions that define dense $\omega$-orders.} In this section we recall a few fundamental results about expansions that are not type A. The following is a minor modification of Hieronymi and Tychonievich \cite[Theorem A]{HT}.

 \begin{fact}[{\cite[Proposition 3.8]{FHW-Compact}}]\label{prop:htplus}
If $\Cal R$ defines a linear order $(D,\prec)$, an open interval $I\subseteq\R$, and a function $g\colon\R^3\times D\rightarrow D$ such that
\begin{itemize}
\item[(i)] $(D,\prec)$ has order type $\omega$ and $D$ is dense in $I$,
\item[(ii)] for every $a,b\in U$ and $e,d\in D$ with $a<b$ and $e\preceq d$,
\[\{c\in\R:g(c,a,b,d)=e\}\cap(a,b)\mbox{ has nonempty interior,}\]
\end{itemize}
then $\Cal R$ is type C.
\end{fact}

\noindent By \cite[Theorem B]{HW-Monadic} every expansion of $(\R,<,+)$ that defines a dense $\omega$-orderable set, interprets the monadic second order theory of one successor. Fact~\ref{prop:htplus} roughly states that we can not expand $\Cal R$ too much without it becoming type C, once we have this $\omega$-orderable set. This should be compared to a similar result of Elgot and Rabin \cite[Theorem 1]{ElgotRabin} on decidable expansions of the monadic second theory of one successor (see \cite[Section 3.2]{BH-Cantor} for a more detailed discussion). The following corollary to Fact \ref{prop:htplus} is often easier to apply.

\begin{fact}[{\cite[Lemma 3.7]{BH-Cantor}}] \label{fact:diff}
If there exist two dense $\omega$-orderable subsets $C$ and $D$ of $[0,1]$ satisfying $(C-C)\cap (D-D)=\{0\}$, then $\mathcal{R}$ is type C.
\end{fact}

\section{Defining a field}
In this section we explain how to recover a field from a non-affine function. In the case of a $C^2$ function our central argument has already appeared in \cite{MPP}. Our main contribution is the extension to $C^1$ functions, in particular statement (2) in the following theorem.

\begin{thm}\label{thm:field-type}
Let $f : I \to \mathbb{R}$ be definable, $C^1$, and non-affine.
\begin{enumerate}
\item If $f'$ is strictly increasing or strictly decreasing on some open subinterval of $I$, then $\Cal R$ is field-type.
\item If $f'$ is not strictly increasing or strictly decreasing on any open subinterval of $I$, then $\Cal R$ defines an isomorphic copy of $( \Cal P(\mathbb{N}), \mathbb{N}, \in, +, \cdot)$.
\end{enumerate}
In particular, if $f$ is $C^2$, then $\mathcal{R}$ is field-type.
\end{thm}

\noindent Theorem F follows. The conclusion of statement (2) in Theorem \ref{thm:field-type} does not rule out the possibility that $\Cal R$ is type A. Note that a structure defines an isomorphic copy of $( \Cal P(\mathbb{N}), \mathbb{N}, \in, +, \cdot)$ if and only if it defines an isomorphic copy of  $(\R,<,+,\cdot,\mathbb{Z})$. Friedman, Kurdyka, Miller, and Speissegger \cite{FKMS} gave an example of type A structure that defines an isomorphic copy of  $(\R,<,+,\cdot,\mathbb{Z})$.\newline

\noindent Before proving Theorem \ref{thm:field-type} we establish a few lemmas used in the proof. We fix one further notation: A \textbf{complementary interval} of $A \subseteq \R$ is a connected component of the complement of the closure of $A$.

\begin{lem}\label{lem:monoton}
Let $F \subseteq \R$ be such that $(F,<)$ is isomorphic to $(\R,<)$.
Then $F$ either has interior or is nowhere dense.
Furthermore, if $I$ is a bounded complementary interval of $F$, then either the left endpoint or the right endpoint of $I$ is in $F$.
\end{lem}

\begin{proof}
Let $\iota : (\R,<) \to (F,<)$ be an isomorphism.
Let $J$ be an open interval.
We suppose that $F$ is dense in $J$ and show $J \subseteq F$. The first claim then follows.
Let $t \in J$ and $X := \{ x \in \R : \iota(x) < t\}$.
The density of $F$ in $J$ yields an $x \in \R$ satisfying $\iota(x) > t$. Thus $X$ is bounded from above.
Let $u$ be the supremum of $X$ in $\R$.
As $F$ is dense in $J$, we must have $\iota(u) = t$. Therefore $t \in F$.
We proceed to the second claim. Let $I$ be a bounded complementary interval of $F$.
The density of $(F,<)$ shows that $F$ contains at most one endpoint of $I$.
Let $z \in I$ and $Y := \{ x \in \R \ : \ \iota(x)<z\}$. As $I$ is a bounded complementary interval, there is an $x \in \R$ such that $\iota(x) > z$. Hence $Y$ is bounded above in $\R$.
Let $u\in \R$ be the supremum of $Y$. If $u \in Y$, then $\iota(u)$ is the left endpoint of $I$. If $u \notin Y$, then $\iota(u)$ is the right endpoint of $I$.
\end{proof}

\begin{lem}\label{lem:endpoint}
Let $A \subseteq \R$ be definable and bounded.
Then the set $D$ of endpoints of bounded complementary intervals of $A$ is $\omega$-orderable.
\end{lem}

\begin{proof}
The statement trivially holds when $D$ is finite.
We now consider the case that $D$ is infinite, and define an $\omega$-order $\prec$ on $D$.
Note that each element of $D$ is the endpoint of at most two bounded complementary intervals.
Let $\delta: D \to \R$ be the definable function that maps each $d$ to the minimal length of a complementary interval with endpoint $d$.
We declare $d \prec d'$ if either $\delta(d') < \delta(d)$ or ($\delta(d') = \delta(d)$ and $d < d'$).
It is easy to see that $\prec$ is an $\omega$-order on $D$ (see Section 2 of \cite{HW-Monadic} for details).
\end{proof}

\begin{lem}\label{prop:interpret}
Let $F \subseteq \R$ be definable and bounded, and $\oplus, \otimes : F^2 \to F$ be definable such that $( F , <, \oplus, \otimes)$ is isomorphic to $(\R,<,+,\cdot)$. Then $F$ either has interior or is nowhere dense and,
\begin{enumerate}
\item if $F$ has interior, then $\Cal R$ is field-type.
\item if $F$ is nowhere dense, then there is a definable $Z \subseteq F$ such that the structure $(F,<,\oplus, \otimes, Z)$ is isomorphic to $( \R,<,+,\cdot,\Z)$.
\end{enumerate}
\end{lem}

\begin{proof}
Lemma~\ref{lem:monoton} shows that $F$ either has interior or is nowhere dense.
Item (1) above follows easily from the fact that for every open interval $I$ there are $(\R,<,+,\cdot)$-definable $\oplus',\otimes' : I^2 \to I$ such that $(I,<,\oplus',\otimes')$ is isomorphic to $(\R,<,+,\cdot)$.
We leave the details of (1) to the reader and prove (2).
Suppose $F$ is nowhere dense.
 Let $D$ be the set of endpoints of bounded complementary intervals of $F$ and $D' = D \cap F$.\newline

\noindent We first show that $D'$ is dense in $F$. Let $x,y \in F$ and $x < y$. Since $F$ is nowhere dense, there is a complementary interval $I$ of $F$ such that $x < z < y$ for every $z\in I$. By Lemma~\ref{lem:monoton} one of the endpoints of $I$ lies in $F$.
Thus $D'$ is dense in $F$.\newline

\noindent Let $\prec$ be the $\omega$-order on $D$ given by Lemma~\ref{lem:endpoint} and denote its restriction to $D'$ by $\prec'$.
Note that $(D',\prec')$ has order-type $\omega$.
Consider $\Cal F:=(F,<,\oplus,\otimes,D',\prec')$. Clearly $\Cal F$ is definable. Note that $\iota$ is an isomorphism between $\Cal F$ and an expansion of $(\mathbb{R},<,+,\cdot)$ that admits a dense $\omega$-orderable set.
An application of Fact~\ref{thm:H} shows that $\Cal F$ defines $Z := \iota^{-1}(\Z)$ and that $(F,<,\oplus,\otimes,Z)$ is isomorphic to\linebreak $(\R,<,+,\cdot,\Z)$.
\end{proof}

\begin{lemma}\label{claim1}
Let $f: [a,b] \to \R$ be $C^1$ and definable, let $\{ \left(g_x : [0,c_x] \to \R\right) : x \in X \}$ be a definable family of functions such that
\begin{enumerate}
\item  $f'(a) = 0$ and $f'(t) > 0$ for all $a < t \leq b$, and
\item $c_x > 0$, $g_x(0) = 0$, and $g'_x(0)$ exists for all $x \in X$.
\end{enumerate}
Then the relations $g'_y(0) < g'_x(0),$ $g'_x(0) \leq g'_y(0)$, and $g'_x(0) = g'_y(0)$ are definable on $X$.
\end{lemma}
\begin{proof}
 We only prove the first claim, the latter two follow. Fix $x,y\in X$. We show that the following are equivalent:
\begin{itemize}
\item[(i)] $g'_x(0) < g'_y(0)$,
\item[(ii)] there is $z \in (a,b)$ such that
\begin{equation}\label{claim:eq}
\tag{$\star$} \quad  g_x(\epsilon)  + [ f( z + \epsilon ) - f(z) ] < g_y(\epsilon) \quad \text{for sufficiently small }\epsilon > 0 .
\end{equation}
\end{itemize}
Suppose (ii) holds. Let $z \in (a,b)$ be such that $\eqref{claim:eq}$ holds. Dividing by $\epsilon$ and taking the limit as $\epsilon \to 0$, we get
\[
g'_x(0) + f'(z) \leq g'_y(0).
\]
As $z > a$, we have $f'(z) > 0$. Thus $g'_x(0) < g'_y(0)$ and (i) holds.\newline

\noindent Suppose (i) holds. Let $\delta > 0$ be such that $g'_x(0) + \delta < g'_y(0)$.
Since $f'$ is continuous and $f'(a) = 0<f'(b)$, there is $z \in (a,b)$ such that $f'(z) < \delta$. Fix such $z$.
Then $g'_x(0) + f'(z) < g'_y(0)$. Thus
$$ \lim_{\epsilon \to 0} \frac{ g_x(\epsilon)  }{\epsilon} + \lim_{\epsilon \to 0} \frac{ f(z + \epsilon) - f(z) }{\epsilon} < \lim_{\epsilon \to 0} \frac{ g_y(\epsilon)}{\epsilon}.$$
Hence $$  \frac{ g_x(\epsilon)  }{\epsilon} +  \frac{ f(z + \epsilon) - f(z) }{\epsilon} < \frac{ g_y(\epsilon)}{\epsilon} \quad \text{for sufficiently small }\epsilon > 0 .$$
Therefore $\eqref{claim:eq}$ holds.
\end{proof}

\begin{proof}[Proof of Theorem \ref{thm:field-type}]
There are $a,b\in \R$ with $a<b$ such that $f'(a) \neq f'(b)$ and one of the following two cases holds:
\begin{itemize}
\item[(I)] $f'$ is strictly increasing or strictly decreasing on $[a,b]$,
\item[(II)] there is no open subinterval of $[a,b]$ on which $f'$ is strictly increasing or strictly decreasing.
\end{itemize}
Now replace $f$ by its restriction to $[a,b]$. So from now on, $f$ is a function from $[a,b]$ to $\R$ satisfying $f'(a) \neq f'(b)$ and either condition (I) or (II).
After replacing $f$ with $-f$ if necessary, we suppose that $f'(a) < f'(b)$.
In case (I) these assumptions imply $f'$ is strictly increasing.\newline

\noindent
Let $q\in \Q$ be such that $f'(a) < q < f'(b)$. By continuity of $f'$ there is an $x\in (a,b)$ such that $f'(x) = q$.
Continuity of $f'$ further implies that the set of such $x$ is closed. Let $c$ be the maximal element of $[a,b]$ such that $f'(c) = q$.
Note that $c < b$. By the intermediate value theorem, $f'(x)>q$ for all $x\in (c,b]$. After replacing $a$ with $c$ if necessary, we may assume that $f'(a) = q$ and $f'(x) > q$ for all $a < x \leq b$.\newline

\noindent Let $h : [a,b] \to \mathbb{R}$ be given by $h(x) = f(x) - q(x - a)$.
Note that $h'(x) = f'(x)  - q$ for all $x\in [a,b]$. Thus $h'(a) = 0$ and $h'(x) > 0$ for all $a < x \leq b$.
Moreover, $h'$ is strictly increasing if $f'$ is. Thus $h'$ is strictly increasing in case (I).
Since $q$ is rational, $h$ is definable. After replacing $f$ with $h$ if necessary, we may suppose that $f'(a) = 0$.\newline

\noindent Let $N\in \N$ satisfy $f'(b) \geq \frac{1}{N}$. After replacing $f$ with $Nf$ if necessary, we can assume  $f'(b) \geq 1$.
Let $d$ be the minimal element of $[a,b]$ such that $f'(d) = 1$. After replacing $b$ with $d$ if necessary, we may suppose that $f'(b) = 1$ and $0 < f'(x) < 1$ for all $a < x < b$.\newline

\noindent  Applying Lemma~\ref{claim1} to the definable family $g_x(t) = f(x + t) - f(x)$ we see that the relations $f'(x) < f'(y)$, $f'(x) \leq f'(y)$, and $f'(x) = f'(y)$ are definable on $I$. Let $E \subseteq [a,b]$ be the set of $x$ such that $f'(y) < f'(x)$ for all $y \in [a, x)$.
Observe that  $E$ is definable.
For every $t\in [0, 1]$ the set $\{ z\in [a, b] \ : \ f'(z) \geq t \}$ is closed and nonempty. Therefore this set has a minimal element $w$. Since $f'(a)=0$ and $f'$ is continuous, this minimal element $w$ must satisfy $f'(w) = t$. In particular, $w \in E$.
Thus for every $t\in [0,1]$ there is an $x \in E$ such that $f'(x) = t$. Note that $a,b \in E$. Furthermore, if $x, y \in E$ and $x < y$, then $f'(x) < f'(y).$
It follows that $x \mapsto f'(x)$ gives an isomorphism between $(E,<)$ and $([0,1],<)$.\newline

\noindent In case (I) we trivially have $E =[a,b]$, because in this case $f'$ is strictly increasing. If $E$ contains an open interval, then $f'$ must be strictly increasing on that interval. Thus $E$ has empty interior in case (II).\newline

\noindent For $x \in E \setminus \{b\}$ and $t\in [0, b - x]$ we set $f_x(t) = f(x + t) - f(x)$.
We declare $f_b(t) = t$ for all $t > 0$. Then $f_x'(0) = f'(x)$ for all $x \in E.$
As $f$ is strictly increasing, each $f_x$ is strictly increasing.
We suppose $a = 0$ after translating $[a,b]$ if necessary. Then $E$ is a subset of $[0,b]$.
We declare
$$ E_1 = -(  E \setminus \{0,b\}) + 2b , \quad E_2 = - (E \setminus \{0\}),  \quad E_3 = -E_1.$$
Then $E,E_1,E_2,E_3$ are pairwise disjoint as they are subsets of $[0,b], (b,2b), [-b,0)$, and $(-2b,-b)$ respectively.
Set $F = E \cup E_1 \cup E_2 \cup E_3.$
Note that in case (I) we have $F = (-2b,2b)$. So $F$ is an interval in this situation.
Furthermore, in case (II) the set $F$ has empty interior as $E$ and each $E_i$ have empty interior.
We now construct a definable family of functions $\{ h_x : x \in F \}$ with the following two properties:
\begin{enumerate}
\item[(i)] For all $t \in \mathbb{R}$ there is a unique $x \in F$ such that $h'_x(0) = t$.
\item[(ii)] If $x,y \in F$ and $x < y$, then $h'_x(0) < h'_y(0)$.
\end{enumerate}
When $x \in E$, we set $h_x = f_x$. When $x \in E_1$, set $h_x$ to be the compositional inverse of $f_{2b - x}$.
Since each $f_x$ is strictly increasing, each $f_x$ has a compositional inverse. Observe that $h'_x(0) = f'_{2b-x}(0)^{-1}$ for all $x \in E_1$.
It follows that for every $t > 1$ there is a unique $x \in E_1$ such that $h'_x(0) = t$.
When $x \in E_2$, we let $h_x = -f_{ -x }$. Then $h'_x(0) = - f'_{-x}(0)$ for all $x \in E_2$.
Again we directly deduce that for every $t \in [-1,0]$ there is a unique $x \in E_2$ such that $h'_x(0) = t$.
Finally, if $x \in E_3$, we set $h_x = -h_{-x}$. Also in this situation we get that for all $t < -1$ there is a unique $x \in E_3$ such that $h'_x(0) = t$.
Conditions (i) and (ii) above follow.\newline

\noindent We are ready to define the field structure on $F$. For this, we need to define two functions $\oplus,\otimes : F^2 \to F$.
Given $x,y \in F$, we let $x \oplus y$ be the unique element of $F$ such that
$$ h'_{ x \oplus y }(0) = ( h_x + h_y )'(0)$$
and $x \otimes y$ be the unique element of $F$ such that
$$ h'_{ x \otimes y} (0) = ( h_x \circ h_y )'(0) . $$
It follows easily from Lemma ~\ref{claim1} that $\oplus$ and $\otimes$ are definable.
By our construction, we immediately get that for all $x,y \in F$
$$ h'_{ x \oplus y}(0) = h'_x(0) + h'_y(0) \text{   and   } h'_{ x \otimes y}(0) = h'_x(0)h'_y(0). $$
So $x \mapsto h'_x(0)$ gives an isomorphism $(F,<,\oplus,\otimes) \to (\mathbb{R},<,+,\cdot)$.
As observed above, $F$ is an interval in case (I) and has empty interior in case (II).
Now apply Lemma~\ref{prop:interpret}.
\end{proof}

\noindent We record some corollaries.

\begin{cor}\label{cor:laffine}
Let $f : I \to \R$ be a non-affine $C^1$ and generically locally affine function. Then
\begin{enumerate}
\item $( \R, <, +,f)$ defines an isomorphic copy of $( \Cal P(\mathbb{N}), \mathbb{N}, \in, +, \cdot)$.
\item $(\R,<,+,\cdot,f)$ is type C.
\end{enumerate}
\end{cor}

\begin{proof}
The derivative of $f$ is locally constant almost everywhere and therefore is not strictly increasing or strictly decreasing on any open subinterval of $I$.
Thus  $(\R,<,+,f)$ defines an isomorphic copy of $( \Cal P(\mathbb{N}), \mathbb{N}, \in, +, \cdot)$ by Theorem~\ref{thm:field-type}. Thus (1) holds.\newline

\noindent For (2), first observe that $f'$ is definable in $(\R,<,+,\cdot, f)$. Let $(F,<,\oplus,\otimes)$ be constructed from $f$ as in the proof of Theorem ~\ref{thm:field-type}. Since $F$ is nowhere dense, we obtain by Lemma \ref{prop:interpret} a $(\R,<,+,\cdot, f)$-definable set $Z\subseteq F$ such that the ordered field isomorphism expands to an isomorphism between $(F,<,\oplus,\otimes,Z)$ and \linebreak $(\R,<,+,\cdot,\mathbb{Z})$.
An inspection of the proof of Theorem ~\ref{thm:field-type} shows that the isomorphism $(F,<,\oplus,\otimes,Z) \to (\R,<,+,\cdot,\mathbb{Z})$ given by $x \mapsto h'_x(0)$ is $(\R,<,+,\cdot,f)$-definable. It follows that $(\R,<,+,\cdot,f)$ is type C.
\end{proof}

\noindent The following corollary shows in particular that if $\Cal R$ is type B and does not define isomorphic copy of the standard model of second order arthimetic, then every definable $C^1$ function $f : I \to \R$ is affine.
This is a special case of Question~\ref{qst:type-B}.

\begin{cor}\label{cor:cantor1}
Let $K$ be a subfield of $\R$ such that $\Cal R$ defines a dense $\omega$-orderable subset of $K$.
Then
\begin{enumerate}
\item if $\Cal R$ is not type C, then every definable $C^2$ function $f : I \to \R$ is affine with slope in $K$.
\item if $\Cal R$ does not define an isomorphic copy of $( \Cal P(\mathbb{N}), \mathbb{N}, \in, +, \cdot)$, then every definable $C^1$ function $f : I \to \R$ is affine with slope in $K$.
\end{enumerate}
\end{cor}

\begin{proof}
First observe that if $\Cal R$ does not define an isomorphic copy of $( \Cal P(\mathbb{N}), \mathbb{N}, \in, +, \cdot)$, then $\Cal R$ is not type C. Since $\Cal R$ defines a dense $\omega$-orderable set, it has to be type B. Therefore $\Cal R$ is not field-type by Fact~\ref{thm:H1}.
Thus by Theorem \ref{thm:field-type} every definable $C^2$ function $f : I \to \R$ is affine.
Moreover, if in addition $\Cal R$ does not define an isomorphic copy of $( \Cal P(\mathbb{N}), \mathbb{N}, \in, +, \cdot)$, then every definable $C^1$ function $f : I  \to \R$ is affine by Theorem \ref{thm:field-type}. It is left to show that the slope of a definable affine function $f : I \to \R$ is in $K$. Towards a contradiction suppose there is such a function with slope $\alpha\notin K$. Its definability immediately implies definability of $x \mapsto \alpha x$ on $[0,1]$.
Let $D$ be a dense $\omega$-orderable subset of $K$. Note that $\alpha D$ is also dense in $\R$.
Observe $D - D \subseteq K$ and $\alpha ( D - D ) \subseteq \alpha K$.
Since $\alpha \notin K$, we have $\alpha K \cap K = \{ 0 \}$. This yields
$$ ( D - D ) \cap (\alpha D - \alpha D )= \{  0 \}.$$
Thus $\Cal R$ is type C by Fact~\ref{fact:diff}. A contradiction.
\end{proof}

\noindent Let $C$ be the middle-thirds Cantor set, or one of the generalized Cantor sets discussed in \cite{BH-Cantor}.
It is observed in the proof of \cite[Corollary 3.10]{FHW-Compact} and the introduction of \cite{BH-Cantor} that $(\R,<,+,C)$ defines a dense $\omega$-orderable subset of $\Q$. Since $(\R,<,+,C)$ has a decidable theory, it can not define an isomorphic copy of $( \Cal P(\mathbb{N}), \mathbb{N}, \in, +, \cdot)$. Therefore Corollary~\ref{cor:cantor1} shows that if $f : I \to \R$ is $C^1$ and $(\R,<,+,C,f)$ does not define an isomorphic copy of $( \Cal P(\mathbb{N}), \mathbb{N}, \in, +, \cdot)$, then $f$ is affine with rational slope.

\section{The one-variable case of Theorem B}
In this section we prove Theorem B for one-variable functions $f : I \to \R$.

\begin{theorem}
\label{thm:A-one-Var}
A continuous definable function $f : I \to \R$ is generically $C^k$ for any $k \geq 1$.
\end{theorem}

\noindent Theorem~\ref{thm:A-one-Var} and Theorem~\ref{thm:dsig-cont} together show that if $f : U \to \R$ is $\DSig$, where $U \subseteq \R$ is open and definable, then $f$ is generically $C^k$.\newline

\noindent The reader will find it helpful to have copies of \cite{FHW-Compact, HW-Monadic} handy, as we repeatedly make use of results from these papers. We need to include a remark about the work in \cite{HW-Monadic}. By \cite[Theorem B]{HW-Monadic}, an expansion of $(\R,<,+)$ that does not define an isomorphic copy of $(\Cal P(\N),\N,\in,+1)$ is type A. In Sections 3 and 4 of \cite{HW-Monadic} all results were stated for expansions that do not define $(\Cal P(\N),\N,\in,+1)$. However, the proofs only made use of the fact that such structures are type A. This should have been made clear, but the authors did not anticipate the relevance of the weaker assumption.

\subsection{Prerequisites}
Throughout this subsection $\Cal R$ is type A.
Before diving into the proof we establish a few basic facts for later use.

\begin{lem}\label{finiteline}
Let $X \subseteq I \times \mathbb{R}_{>0}$ be $\DSig$ such that $X_x$ is finite for every $x \in I$.
Then there is a nonempty open $J \subseteq I$ and $\epsilon > 0$ such that $J \times [0, \epsilon]$ is disjoint from $X$.
\end{lem}

\begin{proof}
Let $\pi : I \times \mathbb{R}_{>0}\to \R$ be the projection onto the first coordinate.
By Fact~\ref{fact:basicdsig} $\pi(X)$ is $D_\Sigma$. Therefore $\pi(X)$ either has interior or is nowhere dense by SBCT.
If $\pi(X)$ is nowhere dense, then there is an open subinterval $J \subseteq I$ that is disjoint from $\pi(X)$. For this subinterval $J$ we get that $J \times \R_{\geq 0}$ is disjoint from $X$. \newline

\noindent Now suppose that $\pi(X)$ has interior. Let $I'$ be an open subinterval of $I$ contained in $\pi(X)$.
After replacing $I$ with $I'$ and $X$ with $X \cap [I' \times \R]$, we may suppose that $\pi(X) = I$.
 Let $\{ B_{s,t} : s,t \in \mathbb{R}_{>0} \}$ be a definable family of compact sets witnessing that $X$ is $\DSig.$
Let
$$ C_{s,t} = \pi( X \setminus B_{s,t} ) \quad \text{and} \quad D_{s,t} = I \setminus C_{s,t} \quad \text{for all } s,t > 0. $$
As $\pi(X)=I$, we have $x \in D_{s,t}$ if and only if $X_x \subseteq (B_{s,t})_x$.
Since each $X_x$ is finite, every $x \in I$ is contained in some $D_{s,t}$. Thus $\bigcup_{s,t} D_{s,t}= I$.
By the classical Baire Category Theorem there are $s,t \in \R_{>0}$ such that  $D_{s,t}$ is somewhere dense. Fix such $s$ and $t$.
Then $X \setminus B_{s,t}$ is the intersection of a $D_\Sigma$ set by an open set and is thus $D_\Sigma$. Hence $C_{s,t}$ is $D_\Sigma$ as well.
Because $D_{s,t}$ is somewhere dense and the complement of a $\DSig$ set, $D_{s,t}$ has interior by SBCT. Let $J$ be an open subinterval whose closure is contained in the interior of $D_{s,t}$.
Then
\[
X \cap \left[ \cl(J) \times \mathbb{R}_{>0} \right] = B_{s,t} \cap \left[ \cl(J) \times \mathbb{R}_{>0} \right].
\]
As $\cl(J) \times \{0\}$ and $B_{s,t}$ are disjoint compact subsets of $\mathbb{R}^2$, there is an $\epsilon > 0$ such that no point in $B_{s,t}$ lies within distance $\epsilon$ of any point in $\cl(J) \times \{ 0 \}$. For such an $\epsilon$ the set
$J \times [0, \epsilon]$ is disjoint from $B_{s,t}$, and thus disjoint from $X$.
\end{proof}

\begin{defn} A subset $D$ of $\R_{>0}$ is a \textbf{sequence set} if it is bounded and discrete with closure $D\cup \{0\}$.
\end{defn}

\noindent It is easy to see that $(D,>)$ has order type $\omega$ when $D$ is sequence set. By \cite[Lemma 3.2]{HW-Monadic} our expansion $\Cal R$ either defines a sequence set or every bounded nowhere dense definable subset of $\R$ is finite.

\begin{lem}\label{lem:unionomega} Let $D$ be a definable sequence set and let $X\subseteq D\times \R$ be definable such that $X_d$ is nowhere dense for each $d\in D$. Then $\bigcup_{d \in D} X_d$ is nowhere dense.
\end{lem}
\begin{proof}
We first write $\bigcup_{d \in D} X_d$ as an increasing union.
Set
\[
Y := \{ (d,x)\in D \times \R \ : \ \exists e \in D \ e \geq d \wedge (e,x) \in X\}.
\]
Observe that $\bigcup_{d \in D} X_d = \bigcup_{d\in D} Y_d$. Because $Y_d \subseteq Y_e$ when $d \geq e$, the family $\{Y_d \ : \ d \in D\}$ is increasing.
As $(D,>)$ has order type $\omega$, the set $\{ e \in D \ : \ d \leq e\}$ is finite for every $d\in D$. Therefore $Y_d$ is a finite union of nowhere dense sets, and hence nowhere dense for every $d \in D$.  By \cite[Lemma 3.3]{HW-Monadic} the set $\bigcup_{d\in D} Y_d$ is nowhere dense. It follows directly that $\bigcup_{d \in D} X_d$ is nowhere dense.
\end{proof}

\subsection{Proof of Theorem~\ref{thm:A-one-Var}}
Throughout this subsection $\Cal R$ is type A.
Let $I = (a,b)$, $f :I \to \R$,  and $h=(h_1,\dots,h_k) \in \R^k$. We define the \textbf{generalized $k$-th difference of $f$} as follows:
\[
\Delta^0 f(x) := f(x),
\]
and for $k\geq 1$
\[
\Delta_{h}^k f(x) := \Delta^{k-1}_{(h_1,\dots,h_{k-1})} f(x+h_k) - \Delta^{k-1}_{(h_1,\dots,h_{k-1})} f(x).
\]
Observe that $\Delta_{(l_1,l_2)}^{k} f(x) = \Delta^1_{l_2} \Delta^{k-1}_{l_1} f(x)$ whenever $l_1 \in \R^{k-1}$ and $l_2 \in \R$. Note that for given $h$, the function $\Delta_{h}^k f$ is defined on the interval $(a,b-k\|h\|)$.\newline

\noindent In the proof of the o-minimal case of Theorem~\ref{thm:A-one-Var} in \cite{LasStein}, one only has to consider the usual $k$-th difference (that is the case when $h_1=\dots=h_n$). Our proof of Theorem~\ref{thm:A-one-Var} however depends crucially on allowing the $h_i$ to differ. The reason for this difference between the two proofs is that the o-minimal frontier inequality is applied in \cite{LasStein}, and this inequality doesn't hold for type A expansions in general.\newline

\noindent Let $J$ be an open subinterval of $I$ and $k\in \N$. A tuple $(u,x) \in \R^{k}_{\geq 0} \times J$ is \textbf{$(J,k)$-suitable} if $x + k\|u\| \in J$. We denote the set of such pairs by $S_{J,k}$. Note that $S_{J,k}$ is open and $\Delta_h^k f(x)$ is defined for each $(h,x)\in S_{J,k}$.\newline

\noindent If $I=(a,b)$ and $(x,h)\in I \times \R_{>0}$, then $((h,\dots,h),x) \in S_{I,k}$ if and only if $a < x < x+kh < b$.\newline

\noindent The following fact about generalized $k$-th differences follows easily by applying induction to $k$. We leave the details to the reader.

\begin{lem}\label{lem:kthdiff}  Let $k \in \N$, $h=(h_1,h_2) \in \R \times \R^{k-1}$ and $x \in \R$ such that $(h,x) \in S_{I,k}$. Let $f,g :I \to \R$ be two functions. Then
\begin{enumerate}
\item $\Delta_{(h_1,h_2)}^k f(x) = \Delta^{k-1}_{h_2} \Delta^1_{h_1} f(x)$, and
\item $\Delta_h^k (f + g)(x) = \Delta^k_h f(x) + \Delta^k_h g(x)$.
\end{enumerate}
\end{lem}

\begin{defn} We say \textbf{$H_k^f$ holds on $J$} if either
\begin{itemize}
\item $\Delta^k_{(h,\dots,h)} f(x) \geq 0$ for all $(x,h)\in J\times \R_{>0}$ with $((h,\dots,h),x) \in S_{J,k}$ or
\item $\Delta^k_{(h,\dots,h)} f(x) \leq 0$ for all $(x,h)\in J \times \R_{>0}$ with $((h,\dots,h),x) \in S_{J,k}$.
\end{itemize}
\end{defn}

\noindent As in \cite{LasStein} our proof of Theorem~\ref{thm:A-one-Var} is based on the following theorem of Boas and Widder.

\begin{fact}[{\cite[Theorem]{BoasWidder}}]\label{fact:boaswidder} Let $f : I \to \R$ be continuous and $k\geq 2$. If $H^f_k$ holds on $I$, then $f^{(k-2)}$ exists and is continuous on $I$.
\end{fact}

\noindent Before proving Theorem~\ref{thm:A-one-Var} we establish Lemma~\ref{lem:sumargument}. Loosely speaking, it states that in order to show that the generalized $k$-th difference is non-negative on a given set, it is enough to prove that the $k$-th difference is non-negative on a subset whose projection onto the first coordinate is a sequence set.

\begin{lem}\label{lem:sumargument} Let $f : J \to \R$ be a continuous definable function and $D$ be a definable sequence set.
If $\Delta_{(d,h)}^k f(x) \geq 0$ for all $((d,h),x)\in S_{J,k} \cap (D \times \R^{k-1})\times J$, then $\Delta^k_u f(x) \geq 0$ for all $(u,x) \in S_{J,k}$.
\end{lem}
\begin{proof}
By continuity of $f$ and openness of $S_{J,k}$, it is enough to show that $\{ (u,x) \in S_{J,k} \ : \ \Delta^k_u f(x) \geq 0\}$ is dense in $S_{J,k}$. Let $U\subseteq S_{J,k}$ be open. Let $(u_1,u_2,x) \in U$, where $u_1\in \R$ and $u_2 \in \R^{k-1}$. Because $D$ is a sequence set, there are $n\in \N$ and  $d_1,\dots,d_n\in D$ such that $(\sum_{i=1}^n d_i,u_2,x)\in U$. It is left to show the following claim: For every $j \in \{1,\dots, n\}$, $(\sum_{i=1}^j d_i,u_2,x)\in S_{J,k}$ and $\Delta^k_{(\sum_{i=1}^j d_i,u_2)} f(x) \geq 0$.\newline

\noindent First observe that since $\sum_{i=1}^j d_i < \sum_{i=1}^n d_i$ and $(\sum_{i=1}^n d_i,u_2,x) \in S_{J,k}$, we have $(\sum_{i=1}^j d_i,u_2,x)\in S_{J,k}$. We now show the second statement of the claim by applying induction to $j$. For $j=1$, $\Delta^k_{(d_1,u_2)} f(x) \geq 0$ by our assumptions on $D$. So now let $j>1$ and suppose $\Delta^k_{(\sum_{i=1}^{j-1} d_i,u_2)} f(x) \geq 0$. Since $(\sum_{i=1}^{j} d_i,u_2,x) \in S_{J,k}$, it follows immediately that
$(d_j,u_2,x+\sum_{i=1}^{j-1} d_i) \in S_{J,k}$. Thus $\Delta^k_{(d_j,u_2)}f(x +  \sum_{i=1}^{j-1} d_i) \geq 0$ by our assumption on $D$.  Applying Lemma \ref{lem:kthdiff} and using our induction hypothesis we obtain
\begin{align*}
\Delta^k_{(\sum_{i=1}^j d_i,u_2)}& f(x) = \Delta^{k-1}_{u_2} \Delta^1_{\sum_{i=1}^j d_i} f(x) =  \Delta^{k-1}_{u_2} \left( f\left(x+\sum_{i=1}^j d_i\right)- f(x)\right)\\
=&  \Delta^{k-1}_{u_2} \left( f\left(x+\sum_{i=1}^j d_1\right) - f\left(x+\sum_{i=1}^{j-1} d_i\right)+ f\left(x+\sum_{i=1}^{j-1} d_i\right) - f(x)\right)\\
=&  \Delta^{k-1}_{u_2} \left( \Delta^1_{d_j} f\left(x+\sum_{i=1}^{j-1} d_i\right)+ \Delta^1_{\sum_{i=1}^{j-1} d_i}f(x)\right)\\
=&  \Delta^{k-1}_{u_2} \Delta^1_{d_j} f\left(x+\sum_{i=1}^{j-1} d_i\right)+ \Delta^{k-1}_{u_2}\Delta^1_{\sum_{i=1}^{j-1} d_i}f(x)\\
=& \Delta^k_{(d_j,u_2)}f\left(x +  \sum_{i=1}^{j-1} d_i\right)  + \Delta^k_{(\sum_{i=1}^{j-1} d_i,u_2)} f(x) \geq 0.
\end{align*}
\end{proof}

\noindent We are now ready to prove the following stronger version of Theorem~\ref{thm:A-one-Var}. It states that the open dense set on which the continuous function $f$ is $C^k$, can be defined uniformly in the parameters that define $f$.
We will use the stronger form to prove the multivariable case of Theorem B.

\begin{theorem}\label{thm:uniform}
Let $Z\subseteq \R^n$ be definable, let $(I_z)_{z\in Z}$ be a definable family of bounded open intervals, and let $(f_z : I_z \to \R)_{z \in Z}$ be a definable family of continuous functions. Then there is a definable family $(U_z)_{z\in Z}$ of open dense subsets of $I_z$ such that $f_z$ is $C^k$ on $U_z$ for every $z\in Z$.
\end{theorem}

\begin{proof}
For $z\in Z$, let $a_z,b_z\in \R$ be such that $I_z=(a_z,b_z)$. We show that for every $k\in \N$ there is a definable family $(U_z)$ of open dense subsets of $I_z$ such that $f_z$ is $C^k$ on $U_z$ for each $z\in Z$.\newline

\noindent We first treat the case when $\mathcal{R}$ defines a sequence set $D$. By Fact \ref{fact:boaswidder} it is enough to show that for every $k\in \N$ there is a definable family $(U_z)_{z\in Z}$ of open dense subsets of $I_z$ such that for every $z\in Z$ and for every connected component $J$ of $U_z$
\begin{itemize}
\item $\Delta^k_h f_z(x)\geq 0$ for all $(h,x)\in S_{J,k}$, or
\item $\Delta^k_h f_z(x)\leq 0$ for all $(h,x)\in S_{J,k}$.
\end{itemize}
We proceed by induction on $k$. The case $k=1$ follows easily from Fact \ref{fact:weakmonot}.
\newline

\noindent Let $k>1$. Observe that for every $z\in Z$ and $d\in D$, $\Delta_{d,h}^k f_z = \Delta_{h}^{k-1}  \Delta_{d}^{1} f_z$ and that $\Delta_d^1 f_z$ is defined on the interval $(a_z,b_z-d)$. By the induction hypothesis there is a definable family $(U_{z,d})_{
(z,d) \in Z\times D}$ of dense open subsets of $(a_z,b_z-d)$ such that for each connected component $J$ of $U_{z,d}$, either
\begin{itemize}
\item $\Delta_{h}^{k-1} \Delta_{d}^1 f_z(x) \geq 0$ for all $(h,x) \in S_{J,k-1}$, or 
\item $\Delta_{h}^{k-1} \Delta_{d}^1 f_z(x) \leq 0$ for all $(h,x) \in S_{J,k-1}$. 
\end{itemize}
For $(z,d) \in Z \times D$ set
\[
X_{z,d} := \Big((a,b-d) \setminus U_{z,d}\Big) \cup \{b-d\}.
\]
By Lemma \ref{lem:unionomega} the set $\bigcup_{d\in D} X_{z,d}$ is nowhere dense for each $z\in Z$. Set 
\[
U_z := I_z \setminus \cl \left(\bigcup_{d\in D} X_{z,d}\right).
\] 
Observe that $(U_z)_{z\in Z}$ is a definable family of dense open subsets of $I_z$. \newline

\noindent Let $z\in Z$ and let $J$ be a connected component of $U_z$. Then for each $d \in D$, either
\begin{itemize}
\item[(i)] $(a,b-d) \cap J = \emptyset$ or
\item[(ii)] $J\subseteq (a,b-d)$ and one of the following is true:
\begin{itemize}
\item[(a)] $\Delta_{h}^{k-1} \Delta_{d}^1 f_z(x) \geq 0$ for all $(h,x) \in S_{J,k-1}$, or
\item[(b)] $\Delta_{h}^{k-1} \Delta_{d}^1 f_z(x) \leq 0$ for all $(h,x) \in S_{J,k-1}$.
\end{itemize}
\end{itemize}
Since $D$ is a sequence set, there are infinitely many $d\in D$ for which (ii) holds. Denote the set all such $d\in D$ by $D'$.
Let
\[
D'':= \{ d \in D' \ : \ \Delta_{h}^{k-1} \Delta_{d}^1 f(x) \geq 0 \hbox{ for all } (h,x) \in S_{J,k-1}\}.
\]
 Then either $D'\setminus D''$ is infinite or $D''$ is infinite. Suppose $D''$ is infinite. We now want to show that $\Delta^k_u f_z(x)\geq 0$ for all $(u,x)\in S_{J,k}$. By Lemma \ref{lem:sumargument} it is enough to show that $\Delta_{d,h}^k f_z(x) \geq 0$ for all $((d,h),x)\in S_{J,k} \cap (D''\times \R^{k-1})\times J$.
Let $(d,h,x) \in S_{J,k}\cap (D'' \times \R^{k-1})\times J$. By definition of $S_{J,k}$, we get that $x + k\|(d,h)\|\in J$. Thus $x +(k-1)\|h\| \in J$ and hence $(h,x) \in S_{J,k-1}$. Since $d\in D''$, we get
\[
\Delta_{d,h}^k f_z(x)=\Delta_{h}^{k-1} \Delta_{d}^1 f_z(x)\geq 0.
\]
The case when $D' \setminus D''$ is infinite may be handled similarly.\newline

\noindent We now suppose that $\Cal R$ does not define a sequence set. By \cite[Lemma 3.2]{HW-Monadic} every bounded nowhere dense definable subset of $\R$ is finite. Set
\[
S_z:=\{ (x,h) \in I_z \times \R_{>0} \ : \ (h,\dots,h,x) \in S_{I_z,k+2}\}.
\]
and
\begin{align*}
V_{1,z} &:=  \{ (x,h) \in S_z \ : \  \Delta^{k+2}_{(h,\dots,h)} f_z(x)\geq 0 \}\\
V_{2,z} &:=  \{ (x,h) \in S_z \ : \  \Delta^{k+2}_{(h,\dots,h)} f_z(x)\leq 0 \}.
\end{align*}
Observe that $S_z$ is open and both $V_{1,z}$ and $V_{2,z}$ are closed in $S_z$. Let 
\[
W_z:=S_z\setminus \left(\Int V_{1,z} \cup \Int V_{2,z}\right).
\]
Then $W_z$ is $\DSig$ for each $z\in Z$. Since $W_z\subseteq (V_{1,z}\setminus \Int V_{1,z})\cup (V_{2,z}\setminus \Int V_{2,z})$, we have that $W_z$ is nowhere dense and therefore has no interior.
It follows immediately from \cite[Fact 2.9(1) \& Proposition 5.7]{FHW-Compact}
that $\dim W_z \leq 1$.
Let $\pi : \R \times \R_{>0} \to R$ be the coordinate projection onto the first coordinate.
Consider
\[
Y_z := \{ x \in I \ : \ \dim (W_z)_x = 1\}.
\]
By \cite[Fact 2.14(2)]{FHW-Compact} the set $Y_z$ is $\DSig$ for each $z\in Z$. By \cite[Theorem 3]{FHW-Compact}, we get that $\dim Y_z = 0$. Hence $Y_z$ is nowhere dense. Now let
$U_z$ be the complement of $\cl(Y_z)$ in $I_z$. From the definition of $Y_z$ we obtain that $\dim (W_z)_x =0$ for all $x \in U_z$. In particular, each $(W_z)_x$ is nowhere dense and hence finite.
Consider
\[
V_z := \{ x \in U_z \ : \ \forall \delta,\epsilon >0 \ (x-\delta,x+\delta)\times (0,\epsilon) \cap W_z \neq \emptyset\}.
\]
We will show that $V_z$ is nowhere dense for each $z\in Z$. Suppose $J$ is an open subinterval of $I_z$ in which $V$ is dense. Observe that $(J \times \R_{>0})\cap W_z$ is $\DSig$. Applying Lemma \ref{finiteline} to this set we get a subinterval $J'\subseteq J$ and an $\epsilon > 0$ such that $J' \times (0,\epsilon)$ is disjoint from $W_z$. This contradicts the density of $V$ in $J$. Thus $V$ is nowhere dense.\newline

\noindent Let $U'_z$ be the complement of $\cl(V_z)$. It is left to show that for the each $z\in Z$, the function $f_z$ is $C^k$ on $U'_z$. Let $x\in U'$. As $x\notin V_z$, there are $\delta,\epsilon >0$ such that $(x-\delta,x+\delta)\times (0,\epsilon) \cap W_z = \emptyset$. It follows from connectedness that $(x-\delta,x+\delta)\times (0,\epsilon)$ is contained in $\Int(V_{1,z})$ or $\Int(V_{2,z})$. If necessary decrease $\delta$ so that $2\delta < (k+2)\epsilon$. Then it is easy to check that
$H^{f_z}_{k+2}$ holds on $(x-\delta,x+\delta)$. By Fact \ref{fact:boaswidder} the function $f_z$ is $C^{k}$ on $(x-\delta,x+\delta)$.
\end{proof}

\section{Proof of Theorem A}

In this section we will prove Theorem A. For the convenience of the reader we first recall the statement of the theorem.

\begin{thmA}
Suppose $\Cal R$ is type A.
The following are equivalent:
\begin{enumerate}
\item $\Cal R$ is field-type,
    \item there is a $\DSig$ field $(X,\oplus,\otimes)$ with $\dim X > 0$,
    \item there is a $\DSig$ family $(A_x)_{x \in B}$ of subsets of $\R^n$ such that $\dim B \geq 2$, each $A_x$ is one-dimensional, and $A_x \cap A_y$ is zero-dimensional for distinct $x,y \in B$,
    \item there is a definable open $U \subseteq \R^m$ and a $\DSig$ function $f : U \to \R^n$ that is nowhere locally affine,
    \item there is a $\DSig$ function $f : I \to \R$ that is nowhere locally affine.
\end{enumerate}
\end{thmA}

\noindent We complete the proof in several steps.
We first establish that $(4)$ implies $(5)$.
In fact, we prove the following more general result that does not require the assumption that $\Cal R$ is field-type.

\begin{lem}\label{lem:multivar}
If every continuous definable $f : I \to \R$ is generically locally affine, then every continuous definable $f : U \to \R^m$, where $U \subseteq \R^k$ is open, is generically locally affine.
\end{lem}

\noindent Now Lemma~\ref{lem:multivar} and Theorem~\ref{thm:dsig-cont} give that $(4)$ implies $(5)$. For the proof Lemma~\ref{lem:multivar} we need the following basic fact from analysis.


\begin{fact}\label{fact:afffine2}
A continuous $f : J \to \R$ is affine if and only if
$$ \frac{ f(x) + f(y) }{2} = f \left( \frac{ x + y }{2} \right) \quad \text{for all } x,y \in J.$$
\end{fact}

\noindent We also need a selection theorem for locally closed sets. A set $X\subseteq \R^n$ is \textbf{locally closed} if for every point $x\in X $ there is an open set $U\subseteq \R^n$ containing $x$ such that $X\cap U$ is closed in $U$. Given $C \subseteq \R^n$ and $p \in \R^n$, we let
\[
d(p,C) := \inf \{ \| p - q \| : q \in C \} .
\]
We also let $B_n(q,r)$ be the open ball in $\R^n$ with center $q$ and radius $r > 0$.

\begin{lem}
\label{lem:select}
Let $A \subseteq \R^m \times \R^n$ be definable such that $A_p$ is locally closed for all $p \in \R^m$.
Let $\pi$  be the coordinate projection $\R^m \times \R^n \to \R^m$.
Then there is a definable function $g : \pi(A) \to \R^n$ such that $(p,g(p)) \in A$ for all $p \in \pi(A)$.
\end{lem}

\begin{proof}
We first reduce to the case when $A_p$ is bounded for every $p\in \R^m$.
Let $f : \pi(A) \to \R$ be given by
$$ f(p) = \inf\{ r \in \R_{>0} \ : \ B_n(0,r) \cap A_p \neq \emptyset \} + 1.  $$
Then $\{ B_n(0,f(p)) \cap A_p : p \in \pi(A) \}$ is a definable family of nonempty bounded locally closed sets.
So we may assume that each $A_p$ is bounded.
For each $p \in \pi(A)$ let $W_p$ be the union of all open boxes $B$ of diameter at most one in $\R^n$ such that $B \cap A_p$ is closed in $B$.
Then $\{ W_p : p \in \pi(A) \}$ is a definable family of bounded open sets such that $A_p \subseteq W_p$ and $A_p$ is closed in $W_p$ for all $p \in \pi(A)$.
Let $C$ be the set of $(p,q) \in \R^m \times \R^n$ such that $(p,q) \in A$ and
$$ d(q, \R^n \setminus W_p) = \max \{ d(x,\R^n \setminus W_p) : x \in A_p \}. $$
It is easy to see that $C$ is definable and each $C_p$ is nonempty and compact. Let $g:\pi(A) \to \R^n$ be the function that maps $p \in \pi(A)$ to the lexicographically minimal element of $C_p$.
It is easy to check that $g$ is definable and that $(p,g(p)) \in A$ for all $p \in \pi(A)$.
\end{proof}

\begin{proof}[Proof of Lemma~\ref{lem:multivar}]
Let 
\[
f(x) = (f_1(x),\ldots,f_m(x)) \quad \text{for all  } x \in \R^k.
\]
Suppose that for each $1 \leq i \leq m$ there is an open dense definable subset $U_i$ of $U$ on which $f_i$ is affine.
Then $f$ is affine on $U_1 \cap \ldots \cap U_m$. Thus without loss of generality, we can assume that $m = 1$.
\newline

\noindent
We apply induction on $k$.
The base case $k = 1$ holds by assumption.
Let $B := I_1 \times \ldots \times I_k$ be a box contained in $U$.
We show that there is a nonempty box contained in $B$ on which $f$ is affine. The argument goes through uniformly and therefore shows that $f$ is locally affine on a dense open subset of $U$.
Let $B' = I_1 \times \ldots \times I_{k-1}$ and let $\pi : B \to B'$ be the projection away from the last coordinate. Define $f_x(t) := f(x,t)$ for all $(x,t) \in B' \times I_k$.
For each $\delta > 0$ we define $E_{\delta}$ to be the set of all $(z,t) \in B' \times I_k$ such that $(t - \delta, t + \delta)$ is a subset of $I_k$ on which $f_z$ is affine.
Note that $E_\delta \subseteq E_{\delta'}$ when $\delta' < \delta$.
By Fact~\ref{fact:afffine2} the restriction $f_z$ to $(t - \delta, t + \delta)$ is affine if and only if
$$   \frac{ f_z(x) + f_z(y) }{2} = f_z \left( \frac{ x + y }{2} \right) \quad \text{for all } x,y \in (t - \delta, t + \delta).$$
Continuity of $f$ therefore implies each $E_\delta$ is closed.
Let $E$ be $\bigcup_{\delta > 0} E_\delta$.
Then $E$ is $F_\sigma$ and $E$ is the set of $(z,t) \in B' \times I_k$ such that $f_z$ is locally affine at $t$.
By our assumption $E$ contains a dense open subset of $\{z\} \times I_k$ for all $z \in B'$.
An application of Fact~\ref{fact:KU} shows that $E$ has interior.
After shrinking $B$ if necessary, we can assume that $f_z$ is affine on $I_k$ for all $z \in B'$. \newline

\noindent
Let $\alpha,\beta : B' \to \R$ be such that $f_x(t) = \alpha(x)t + \beta(x)$ for all $x \in B'$ and $t \in I_k$.
We first show that $\alpha$ is constant in $x$. Suppose not.
Let $q \in \Q_{>0}$ and $y,y' \in I_k$ be such that $y' - y = q$.
Note that
$$ \alpha(z) = q^{-1}[ f_z(y') - f_z(y) ] \quad \text{for all  }  z \in B'. $$
Thus $\alpha$ is definable and continuous. Since $\alpha$ is non-constant, the intermediate value theorem yields a nonempty open interval $L$ contained in the range of $\alpha$. \newline

\noindent By continuity of $\alpha$  the set $\{ z \in B' : \alpha(z) = s \}$ is closed in $B'$ for all $s \in L$.
Applying Lemma~\ref{lem:select} we obtain a definable $g : L \to B'$ such that $\alpha(g(s)) = s$ for all $s \in L$.
So $f_{g(s)}$ has slope $s$ for all $s \in L$. \newline

\noindent
Let $r \in I_k$, $r' \in L$, and $\delta > 0$ be such that for all $r - \delta < t < r + \delta$ we have $t \in I_k$ and $t + (r' - r) \in L$.
Let $h : (r - \delta, r + \delta) \to \R$ be given by
$$ h(t) = f_{g(t + r' -r)}(t) - f_{g(t + r' - r)}(r). $$
Then for all $r - \delta < t < r + \delta$ we have
\begin{align*}
    h(t) =&\, [ (t + r' -r)(t) + \beta(g(t+r'-r))] \\
    &\quad -[ (t + r' -r)(r) + \beta(g(t+r'-r))]\\
    =&\, t^2 + t(r'-2r) - r'r + r^2.
\end{align*}
So $h$ is definable and nowhere locally affine.
Contradiction.\newline

\noindent
We have shown that $\alpha$ is constant on $B'$. Let $a\in \R$ be such that $\alpha(z) = a$ for all $z \in B'$.
Therefore $f_x(t) = a t + \beta(x)$ for all $(x,t) \in B' \times I_k$.
Because $\beta(x) = f_x(t) - a x$ for all $(x,t) \in B' \times I_k$, $\beta$ is definable and continuous on $B'$.
By our induction hypothesis, $\beta$ is affine on a box contained in $B'$.
So after shrinking $B'$ if necessary we may assume that $\beta$ is affine on $B'$.
Thus $f$ is affine on $B' \times I_k$.
\end{proof}

\begin{proof}[Proof of Theorem A]
It is clear that $(1)$ implies $(2)$. \newline

\noindent $(2) \Rightarrow (3):$ Let $(X,\oplus,\otimes)$ be a $\DSig$-field such that $\dim X > 0$ and $X$ is a subset of $\R^n$.
Applying Corollary~\ref{cor:to-selection} we obtain a nonempty open interval $I$ and a continuous definable injection $g : I \to X$.
For $(a,a') \in I^2$ we set
$$ A_{(a,a')} := \{ (b,b') \in I \times X : b' = \big( g(a) \otimes g(b) \big) \oplus g(a') \}. $$
Observe that each $A_{(a,a')}$ is the graph of the function $I \to \R^{n}$ given by
$$ x \mapsto [g(a) \otimes g(x)] \oplus g(a'). $$ This function is injective, because $(X,\oplus,\otimes)$ is a field.  By Theorem~\ref{fact:dim-bijection} we know that $A_{(a,a')}$ is one-dimensional for every $(a,a') \in I^2$.
Thus $(A_{(a,a')})_{(a,a') \in I^2}$ is a definable family of one-dimensional subsets of $\R^{n+1}$ with $\dim I\times I=2$. \newline

\noindent Let $(a_1,a_1')$ and $(a_2,a_2')$ be distinct elements of $I^2$. It is left to show that $A_{(a_1,a_1')}\cap A_{(a_2,a_2')}$ is finite. Let $(b,b') \in A_{(a_1,a_1')}\cap A_{(a_2,a_2')}$. Then 
\[
\big(g(a_1') \otimes g(b) \big) \oplus g(a_1') = \big(g(a_2) \otimes g(b)\big) \oplus g(a_2')
\]
Since $(X,\oplus,\otimes)$ is a field and $g$ is injective, there is at most one such $b$. Thus $A_{(a,a')} \cap A_{(b,b')}$ contains at most one element. \newline

\noindent $(3) \Rightarrow (4):$ Let $(A_x)_{x \in B\subseteq \R^l}$ be a $\DSig$ family of subsets of $\R^n$ such that $\dim B \geq 2$, $\dim A_x = 1$ for all $x \in B$, and $\dim A_x \cap A_y = 0$ for distinct $x,y \in B$. Towards a contradiction, suppose that $(4)$ fails.\newline

\noindent Let $A = \{ (x,y) \in B \times \R^n : y \in A_x \}.$ We first show that we assume that $B$ is an open subset of $\R^2$.
By Corollary~\ref{cor:to-selection} there is a nonempty definable open $V \subseteq \R^2$ and a continuous definable injection $g : V \to B$.
Set
\[
A' := \{ (p,q) \in V \times \R^n : (g(p),q) \in A \}.
\]
Then $A'_{p} = A_{g(p)}$ for all $p \in V$.
Because $A'$ is the preimage of $A$ under a continuous definable map, $A'$ is $\DSig$ by Fact~\ref{fact:basicdsig}.
After replacing $A$ with $A'$ and $B$ with $V$, we may assume that $B$ is an open subset of $\R^2$.\newline

\noindent Let $\pi_k : \R^n \to \R$ be the projection onto the $k$-th coordinate for $k=1,\dots,n$.
As each $A_x$ is one-dimensional, Fact~\ref{thm:factsaboutdsig} shows that for all $x \in B$ there is $1 \leq k \leq n$ such that $\pi_k(A_x)$ has interior.
For $k=1,\dots,n$, we set
\[
A_k := \{ x \in B :  \dim \pi_k(A_x) \geq 1 \}.
\]
It follows from \cite[Fact 2.14]{FHW-Compact} that $A_k$ is $D_\Sigma$ for $1 \leq k \leq n$.
Since $B=\bigcup_{k=1}^n A_k$, there is $k$ such that $A_k$ is somewhere dense in $B$.
By SBCT there is $k$ such that $A_k$ has interior. Let us assume that $A_1$ has interior. 
The case that $A_k$ has interior for $k$ with $2 \leq k \leq n$, can be handled similarly.
After replacing $B$ with a definable nonempty open subset of $A_1$ we may suppose that $\dim \pi_1(A_x)\geq 1$ for every $x\in B$.
Let $\rho : \R^{n+2} \to \R^3$ be the projection onto the first three coordinates. Note that $\rho(A)_x = \pi_1(A_x)$ for all $x \in \R^2$.
Thus $\dim \rho(A)_x\geq 1$ for all $x \in B$. Since $\rho(A)$ is $\DSig$, we know that $\dim \rho (A) = 3$ by \cite[Theorem F(3)]{FHW-Compact}.
Thus $\rho(A)$ has interior by Fact~\ref{thm:factsaboutdsig}.\newline

\noindent Let $I,J,L \subseteq \R$ be nonempty open intervals such that $I \times J \times L \subseteq \rho(A)$.
Thus for all $(x,y,z) \in I \times J \times L$ there is an $u \in \R^{n-1}$ such that $(z,u) \in A_{(x,y)}$.
Applying $D_\Sigma$-selection and replacing $I,J,L$ with smaller nonempty open intervals if necessary, we obtain a continuous definable $f : I \times J \times L \to \R$ such that $(z, f(x,y,z)) \in A_{(x,y)}$ for all $(x,y,z) \in I \times J \times L$. By our assumption that (4) fails, we can find open subset $U\subseteq I\times J \times L$ on which $f$ is affine. Replacing $I,J,L$ with even smaller nonempty open intervals, we can assume that $f$ is affine on $I\times J \times L$. Now fix linear $h_1,h_2,h_3 : \R \to \R^{n-1}$ and $\beta \in \R^{n-1}$ such that
$$ f(a_1,a_2,a_3) = h_1(a_1) + h_2(a_2) + h_3(a_3) + \beta \quad \text{for all}  \quad (a_1,a_2,a_3) \in I \times J \times L. $$
Fix $(u,v) \in I \times J$ and $u' \in I$ such that $u' \neq u$.
Let $v' \in J$ satisfy 
\[
 h_1(u')  +h_2(v') = h_1(u) + h_2(v).
 \]
Then $f(u,v,t) = f(u',v',t)$ for all $t \in L$.
So $\{ (t,f(u,v,t)) : t \in L \}$ is a subset of $A_{(u,v)} \cap A_{(u',v')}$.
We attain a contradiction as $A_{(u,v)} \cap A_{(u',v')}$ is zero-dimensional. \newline

\noindent Lemma~\ref{lem:multivar} shows that $(4)$ implies $(5)$.
It remains to establish that $(5)$ implies $(1)$.
Suppose $f : I \to \R$ is $\DSig$ and nowhere locally affine.
After applying the one variable case of Theorem B and shrinking $I$ if necessary, we may assume that $f$ is $C^2$.
Because $f$ is non-affine, $\Cal R$ is field-type by Fact~\ref{fact:c2}.
\end{proof}

\section{Theorem B for multivariable functions}
\noindent
We assume that $\mathcal{R}$ is type A throughout this section. The goal is the proof of Theorem B for multivariable functions.  We begin with an outline of this proof.
By Theorem~\ref{thm:dsig-cont} it suffices to prove Theorem B for continuous definable functions. Let $f : U \to \R^n$ be continuous and definable and 
$U\subseteq \R^m$ open. We first show that generically all partial derivatives of $f$ exist. However, in order to prove continuity of these derivatives, we have to invoke Theorem A. Indeed, by Theorem A, if $\Cal R$ is not field-type, then $f$ is generically locally affine and hence generically $C^\infty$. Therefore it suffices to treat the case when $\Cal R$ is field-type. We further reduce this case to the case that $\Cal R$ is actually an expansion of $(\R,<,+,\cdot).$ In this situation, we can use the definability of the partial derivatives of $f$ to show that these derivatives are continuous almost everywhere.

\begin{lem}
\label{lem:multivariable0}
Let $k \geq 1$, let $U \subseteq \R^n$ be an open definable set, let $f : U \to \R$ be a continuous definable function, and let $i$ be such that $1 \leq i \leq n$.
Then there is a dense open definable set $V \subseteq U$ such that $\frac{\partial^{k} f}{\partial x_i}(p)$ exists for all $p \in V$.
\end{lem}

\noindent In the following proof we write $\Delta^k_h f$ for $\Delta^k_{(h,\ldots,h)} f$.

\begin{proof}
To simplify notation we suppose $i = 1$.
Let $W = I_1 \times \ldots \times I_n$
be a product of closed intervals with nonempty interior such that the closure of $W$ is contained in $U$.
We show that $\frac{\partial^{k} f}{\partial x_1}$ exists and is continuous on a dense definable  open subset of $W$.
Our argument goes through uniformly in $W$, so we obtain a dense definable open subset of $U$ on which $\frac{\partial^{k} f}{\partial x_1}$ exists and is continuous.\newline

\noindent Let $I_1 = [a,b]$ and $B = I_2 \times \ldots \times I_n$. For  $x \in B$ we define $f_x : I_1 \to \R$ to be the function given by $f_x(t) := f(t,x)$ for $t\in I_1$.
Note that $f^{(k)}_x(t) = \frac{\partial^{k} f}{\partial x_1} (t,x)$ (if either exists).
As $f$ is continuous, $\Delta^{k + 2}_h f_x(t)$ is a continuous function on the closed set 
\[
D := \{ (h,t,x) : h \geq 0, (t,x) \in W, t + kh \leq b \}.
\]
Set
\[
C_1 := \{ (h,t,x) : (h,t,x) \in D, \Delta^{k+2}_h f_x(t) \geq 0  \} 
\]
and
\[
 C_2 := \{ (h,t,x) : (h,t,x) \in D, \Delta^{k+2}_h f_x(t) \leq 0\}.
\]
Observe that $C_1$ and $C_2$ are closed definable sets.
For $i \in \{1,2\}$, set
\[
E_i := \{ (\delta, t, x) : \delta \geq 0, (t,x) \in W  \text{ and } [0,\delta] \times [t - \delta, t + \delta] \times \{x\} \subseteq C_i \}.
\]
Both $E_1$ and $E_2$ are closed definable sets.
If $\delta > 0$ and $(\delta,t,x) \in E_1$, then $\Delta^{(k+2)}_h f_x(t)$ exists and is nonnegative on $[t - \delta, t + \delta]$. Thus for such triples $(\delta,t,x)$, the $k$-derivative $f_x^{(k)}$ exists and is continuous on $[t - \delta, t + \delta]$ by Fact \ref{fact:boaswidder}.
Likewise, if $\delta > 0$ and $(\delta,t,x) \in E_2$, then $f^{(k)}_x$ exists and is continuous on $[t - \delta, t + \delta]$.\newline

\noindent Now for $i \in \{1,2\}$, set
\[
F_i := \{ (t,x) \in W : \exists \delta > 0 \quad (\delta,t,x) \in E_i \}.
\]
Note that $F_1$ and $F_2$ are $\DSig$.
Set $F := F_1 \cup F_2$. Note that $F$ is $\DSig$, and that  $f^{(k)}_x$ exists and is continuous in $t$ for every $(x,t) \in F$.
By the proof of Theorem \ref{thm:uniform}, there is a definable family $(U_x)_{x\in B}$ of open dense subsets of $I$ such that $U_x \times \{x\}\subseteq F$ for every $x \in B$. Therefore $F$ contains an open dense subset of $I \times \{x\}$ for all $x \in B$. 
As $F$ is $F_\sigma$, Fact~\ref{fact:KU} shows that $F$ contains a dense open subset of $W$.
Let $V$ be the interior of $F$.
Then $\frac{\partial^{k} f}{\partial x_1}$ exists on $V$.
\end{proof}

\noindent We now establish a lemma allowing us to reduce certain questions about definable sets and functions in field-type expansions to questions about expansions of $(\R,<,+,\cdot)$. It is crucial for our proof of Theorem B, but we anticipate further applications.

\begin{lem}
\label{lem:diffeo}
Fix $k \geq 1$.
Suppose that $\Cal R$ is field-type.
Then there is an open interval $I$, definable function $\oplus,\otimes : I^2 \to I$, an isomorphism $\tau : (I,<,\oplus,\otimes) \to (\R,<,+,\cdot)$, and $J \subseteq I$ such that the restriction of $\tau$ to $J$ is a $C^k$-diffeomorphism $J \to \tau(J)$.
\end{lem}

\begin{proof}
By Theorem A, the expansion $\Cal R$ defines a non-affine $C^2$-function $L \to \R$.
By inspection of the proof of Theorem~\ref{thm:field-type} the reader can check that we can construct an open interval $I$, definable functions $\oplus,\otimes : I^2 \to I$, an isomorphism $\tau : (I,<,\oplus,\otimes) \to (\R,<,+,\cdot)$, $J \subseteq I$, and a definable, continuously differentiable function $f : J \to \R$ such that $\tau(x) = f'(x)$ for all $x \in J$.
After applying Theorem B and shrinking $J$ if necessary, we may assume that $f$ is $C^{k+1}$ on $J$.
Thus $\tau$ is $C^k$ on $J$.
As $\tau = f'$ is strictly increasing, and $f$ is $C^2$, we also have $\tau'(x) = f''(x) > 0$ for all $x \in J$.
By inverse function theorem the inverse $\tau^{-1}$ is $C^k$ on $\tau(J)$.
Therefore $\tau$ is a $C^k$-diffeomorphism $J \to \tau(J)$.
\end{proof}

\noindent We now prove Theorem B for expansions of $(\R,<,+,\cdot)$.

\begin{lem}
\label{lem:multivariable2}
Suppose that $\Cal R$ expands $(\R,<,+,\cdot)$.
Let $U \sub \R^n$ be a definable open set, and let $f : U \to \R^m$ be a continuous definable function.
Then $f$ is $C^k$ almost everywhere for all $k \geq 1$.
\end{lem}

\begin{proof}
Let $f(x) = (f_1(x),\ldots,f_m(x))$ for all $x \in \R^n$.
Suppose that for $i=1,\dots, n$ there is a dense definable open $V_i \subseteq U$ on which $f_i$ is $C^k$.
Then $f$ is $C^k$ on the dense definable open set $V_1 \cap \ldots \cap V_m$.
We therefore suppose $m = 1$.\newline

\noindent
It suffices to show that for $i=1,\dots, n$ there is a dense open definable subset $V_i$ of $U$ on which $\frac{\partial^{k} f}{\partial x_i}$ exists and is continuous. If this is true, then $f$ is $C^k$ on $V_1 \cap \ldots \cap V_n$.
Fix $i$ with $1 \leq i \leq n$.
Applying Lemma~\ref{lem:multivariable0} there is a dense definable open set $W \subseteq U$ such that $\frac{\partial^{k} f}{\partial x_i}(p)$ exists for all $p \in W$.
This is a definable function, because $\Cal R$ expands $(\R,<,+,\cdot)$.
Furthermore $\frac{\partial^{k} f}{\partial x_i}$ is a pointwise limit of a sequence of continuous functions $W \to \R$.
An application of \cite[Corollary 5.3]{FHW-Compact} shows that $\frac{\partial^{k} f}{\partial x_i}$ is continuous on a dense open subset $V$ of $W$.
Since the set of points at which $\frac{\partial^{k} f}{\partial x_i}$ is continuous is definable, we may take $V$ to be definable.
\end{proof}

\begin{proof}[Proof of Theorem B]
The case when $\Cal R$ is not field-type follows by Theorem A.
We therefore suppose $\Cal R$ is field-type.
We show that any $p \in U$ has a neighbourhood $W$ such that the restriction of $f$ to $W$ is $C^k$ almost everywhere.
This is enough as our proof is uniform in $p$.
Fix $p \in U$ for this purpose.
\newline

\noindent
Applying Theorem~\ref{lem:diffeo} we obtain an $I$, definable $\oplus,\otimes : I^2 \to I$, an isomorphism $\tau : (I,<,\oplus,\otimes) \to (\R,<,+,\cdot)$, and $J \subseteq I$ such that the restriction of $\tau$ to $J$ is a $C^k$-diffeomorphism $J \to \tau(J)$. \newline

\noindent
Let $\tau_n : I^n \to \R^n$ be given by
$$ \tau_n(x_1,\ldots,x_n) = (\tau(x_1),\ldots,\tau(x_n)) \quad \text{for all  } (x_1,\ldots,x_n) \in I^n. $$
Then $\tau_n$ restricts to a $C^k$-diffeomorphism $J^n \to \tau(J^n)$.
Let $\Cal S$ be the expansion of $(\R,<,+,\cdot)$ by all subsets of $\R^n$ of the form $\tau_n(X)$ for definable $X \subseteq I^n$.
Then $X \subseteq \R^n$ is $\Cal S$-definable if and only if $\tau^{-1}_n(X) \subseteq I^n$ is definable. A function $g : X \to \R^m$ is $\Cal S$-definable if and only if $\tau^{-1}_m \circ g \circ \tau_n : \tau^{-1}_n(X) \to \R^m$ is definable.\newline

\noindent
Let $W := \tau_n(J^n)$.
After translating $U$ if necessary we suppose $p \in W$.
After shrinking $J$ if necessary we suppose that $W \subseteq U$.
Let $g : W \to \R^m$ be given by $g(x) = (\tau_m \circ f \circ \tau^{-1}_n)(x)$ for all $x \in W$. Lemma~\ref{lem:multivariable2} shows that $g$ is $C^k$ on a dense $\Cal S$-definable open subset $V$ of $W$.
As $f = \tau^{-1}_m \circ g \circ \tau_n$, and $\tau_n$ and $\tau^{-1}_m$ are both $C^k$ this shows that $f$ is $C^k$ on $\tau^{-1}_n(U)$.
Finally $\tau^{-1}_n(U)$ is a dense open definable subset of $W$.
\end{proof}

\section{Linearity in type B expansions}

In this section, we study the linearity of type B expansions. As noted in Question \ref{qst:type-B} in the introduction, although type B expansions are not field-type, we do not know whether every continuous function $f: I \to \R$ definable in a type B expansion is generically locally affine. However, in the following two subsections, we are able to obtain results indicating strong linearity of type B structures. We hope that these results might eventually lead to an affirmative answer of Question \ref{qst:type-B}.  

\subsection{Repetitious functions}
Let $f: I \to \R$.
Recall that we say $f$ is \textbf{repetitious} if for every open subinterval $J\subseteq I$
there are $\delta > 0$, $x,y \in J$ such that $\delta < y - x$ and
\[
f(x+\epsilon) - f(x) = f(y + \epsilon) - f(y) \quad \text{for all } 0 \leq \epsilon <  \delta.
\]
We now show Theorem E, which states that if $\cal{R}$ is type B and $f$ is definable, then $f$  is repetitious. As similar result holds for linear type A structures. Observe that if $f$ is generically locally affine, then $f$ is repetitious. Thus if $\cal R$ is type A and not field type and $f$ is definable, then $f$  is repetitious by Theorem A.

\begin{lem}\label{lem:factcor}
Suppose $\Cal R$ is type B.
Let $D$ be an $\omega$-orderable set that is dense in $I$, let $f: I\to \R$ be definable and continuous, and let $J\subseteq I$ be an open interval on which $f$ is nonconstant. Then there are $d_1,d_2,d_3,d_4 \in J\cap D$ such that
$$ d_1 \neq d_2, d_3 \neq d_4 \quad{and} \quad d_1 - d_2 = f(d_3) - f(d_4). $$
\end{lem}
\begin{proof}
Let $J\subseteq I$ be an open subinterval on which $f$ is not constant. Let $a,b\in \R$ be such that $J=(a,b)$. The intermediate value theorem yields an open interval $J'\subseteq f(J)$. Since $D$ is dense in $I$, we have that $f(D\cap J)\cap J'$ is dense in $J'$. Let $a',b' \in \R$ be such that $(a',b')=J'$. After decreasing $b'$ we assume $b'-a'<b-a$. Then both $(-a + D)\cap (0,b'-a')$ and $-a'+ \big( f(D\cap J)\cap J'\big)$ are dense $\omega$-orderable subsets of $(0,b'-a')$. Applying Fact \ref{fact:diff} to these two dense $\omega$-orderable sets, we obtain $d_1,d_2,d_3,d_4 \in J$ such that $d_1 \neq d_2$, $f(d_3)\neq f(d_4)$ and
\[
 (-a+d_1) - (-a+d_2) = -a' + f(d_3) - \big(-a' + f(d_4)\big)
\]
It follows that $d_3 \neq d_4$ and 
\[
d_1 - d_2 = (-a+d_1) - (-a+d_2) = -a' + f(d_3) - \big(-a' + f(d_4)\big) = f(d_3) - f(d_4).
\]
\end{proof}

\begin{proof}[Proof of Theorem E]
Suppose $\cal{R}$ is type B. Let $f : I \to \R$ be continuous and definable. We need to show that $f$ is repetitious.
Let $U$ be the set of $p \in I$ at which $f$ is locally constant.
Note that the restriction of $f$ to $U$ is repetitious.
Let $V$ be the interior of $I \setminus U$.
It suffices to show for every open interval $L \subseteq V$ that the restriction of $f$ to $L$ is repetitious.
We may therefore assume that there is no open subinterval of $I$ on which $f$ is constant.
Since $\Cal R$ is type B, there is a dense $\omega$-orderable subset $D$ of $I$. We declare
\[
C := \{ (d_1,d_2,d_3,d_4) \in D^4 \ : \ d_1\neq d_2, d_3 \neq d_4\}.
\]
For $d=(d_1,d_2,d_3,d_4) \in C$, let $A_d \subseteq \R$ be the set of all $t\in \R$ such that
\begin{itemize}
\item $t+d_3,t+d_4 \in I$, and
\item $d_1-d_2 = f(t+d_3) - f(t+d_4)$.
\end{itemize}
For each $d\in C$, the set $A_d$ is closed in $I$ by continuity of $f$.
 Let $s >0$ and $J$ be an open subinterval of $I$ such that $t+J \subseteq I$ for all $t \in (0,s)$. Let $r\in (0,s)$. Consider the function $g_r: I \to \R$ that maps $c\in I$ to $f(r+c)$. Applying Lemma \ref{lem:factcor} to $g_r$ we obtain a $d=(d_1,d_2,d_3,d_4) \in C$ such that:
\[
d_1-d_2 = g_r(d_3) - g_r(d_4)=f(r+d_3) - f(r+d_4).
\]
Thus $r \in A_{d}$. Therefore $(0,s) \subseteq \bigcup_{d \in C} A_d$. By the Baire Category Theorem $A_d$ has interior for some $d \in C$. Fix such a $d=(d_1,d_2,d_3,d_4)\in C$ and let $J'$ be an open interval in the interior of $A_d$. Then the function $J' \to \R$ given by $ t \mapsto f(t+d_3)-f(t+d_4)$ is constant. The statement of the theorem follows.
\end{proof}

\subsection{Weak Poles}
In this section we give more restrictions on continuous functions definable in type B expansions.
Our results apply to a more general class of expansions, those that do not admit weak poles.
A \textbf{pole} is a definable homeomorphism between a bounded and an unbounded interval.

\begin{defn}\label{def:weakpole} A \textbf{weak pole} is a definable family $\{ h_{d} : d \in E\}$ of continuous maps $h_{d} : [0,d] \to \R$ such that
\begin{itemize}
\item[(i)] $E\subseteq \R_{>0}$ is  closed in $\R_{>0}$ and $(0,\epsilon) \cap E \neq \emptyset$ for all $\epsilon >0$,
\item[(ii)] there is a $\delta>0$  such that $[0,\delta] \subseteq h_d([0,d])$ for all $d \in E$.
\end{itemize}
\end{defn}

\noindent Corollary~\ref{thm:weakpoleconseq} below shows that $\Cal R$ admits a weak pole whenever it defines a pole. To our knowledge weak poles have not been studied before. 
We first prove Theorem C, which states that if $\Cal R$ is type B, then $\Cal R$ does not define a weak pole.

\begin{proof}[Proof of Theorem C]
Towards a contradiction, suppose $\Cal R$ defines a dense $\omega$-orderable set $(D,\prec)$ and a weak pole $\{ h_{d} : d \in E\}$. Using Fact \ref{prop:htplus} we will show that $\Cal R$ is type C, contradicting our assumption that $\Cal R$ is type B. After rescaling we may assume that $D$ is dense in $[0,1]$ and $[0,1] \subseteq h_d([0,d])$ for all $d \in E$.
Set
\[
Z:=\{(a,b) \in [0,1]^2 \ : \ a < b\}.
\]
Let $\lambda : \R_{>0} \to E$ map $x$ to the maximal element of $(-\infty, x] \cap E$. Let $g: [0,1] \times Z \times D \to D$ map $(c,a,b,d)$ to
\[
 \left\{
                   \begin{array}{ll}
                     d, & \hbox{if $c-a > \lambda(b-a)$;} \\
                     \hbox{$\prec$-minimal $e\in D_{\preceq d}$ s.t. $h_{\lambda(b-a)}(c-a) - e$ is minimal}, & \hbox{otherwise.}
                   \end{array}
                 \right.
\]
We will now show that $g$ satisfies the assumptions of Fact \ref{prop:htplus}. For this, let $a,b \in Z$ and $d,e \in D$ with $e\preceq d$.
As $[0,1] \subseteq h_{\lambda(b-a)}([0,\lambda(b-a)])$, there is $z \in [0,\lambda(b-a)]$ such that $h_{\lambda(b-a)}(z)=e$. Since $D_{\preceq d}$ is finite and $h_{\lambda(b-a)}$ is continuous, there is an open interval $I$ around $z$ such that for each $y\in I$, $e$ is the only element in $D_{\preceq d}$ such that $h_{\lambda(b-a)}(y) - e$ is minimal. Let $c \in (a,b)$ be such that $c-a =z$. It follows immediately from the argument above that $g(x,a,b,d)= e$
for all $x \in (c +I)\cap (a,b)$. Thus (ii) of Fact \ref{prop:htplus} holds for our choice of $g$. Therefore $\Cal R$ is type C.
\end{proof}

\noindent We now prove several results about continuous definable functions in expansions that do not admit weak poles.
These results yield Theorem D.

\begin{prop}\label{weakpole:linearfamily}
Suppose $\Cal R$ does not admit a weak pole.
Then every definable family $\{ f_x : x \in \R^l \}$ of linear functions $[0,1] \to \R$ has only finitely many distinct elements.
\end{prop}

\begin{proof}
Let $\{ f_x   :  x \in \R^l\}$ be a definable family of linear functions $[0,1] \to \R$ that has infinitely many distinct elements.
After replacing each $f_x$ with $|f_x|$, we may assume that each $f_x$ takes nonnegative values.
Let $B = \{ f_x(1) : x \in \R^l\}$ and let $g : B \times [0,1] \to \R$ be given by $g(\lambda,t) = f_y(t)$ for any $y \in \R^l$ with $f_y(1) = \lambda$.
Note that $g$ is definable and $g(\lambda,t) = \lambda t$ for all $(\lambda, t) \in B \times [0,1]$.
We declare
$$ \tilde{g}(\lambda,t) = \lim_{\lambda' \in B, \lambda' \to \lambda} g(\lambda',t) \quad \text{for all } ( \lambda , t ) \in \cl(B) \times [0,1].$$
By continuity we have that $\tilde{g}(\lambda, t) = \lambda t$ for all $( \lambda , t ) \in \cl(B) \times [0,1]$.
After replacing $g$ by $\tilde{g}$ and $B$ by $\cl(B)$, we may suppose that $B$ is a closed and infinite subset of $\mathbb{R}_{\geq 0}$.
One of the following holds:
\begin{itemize}
\item $B$ is unbounded.
\item $B$ has an accumulation point.
\end{itemize}
First suppose that $B$ is unbounded.
Let $\{ h_d : d \in \R_{>0} \}$ be the definable family of functions $h_d : [0,d] \to \R$ given by declaring $h_d(t) = g(\lambda,t)$ where $\lambda$ is the minimal element of $B$ such that $g(\lambda,d) \geq 1$.
Then $h_d(t) \geq d^{-1}t$ for all $t \in [0,d]$. It directly follows that $\{h_d : d \in \R_{>0}\}$ is a weak pole.\newline

\noindent Now suppose $(2)$ holds. Let $\mu$ be an accumulation point of $B$.
We declare
$$ \psi(\lambda,t) := |g(\mu,t) - g(\lambda,t)| = | \mu - \lambda | t \quad \text{for all } \lambda \in B, t \in [0,1].$$
Note that $\psi$ is definable.
Set
$$C := \{ |\mu - \lambda | : \lambda \in B \} = \{ \psi(\lambda, 1) : 0 \leq \lambda \leq 1 \} . $$
 Observe that $C$ is closed,  definable, and contains arbitrarily small positive elements as $\lambda$ is an accumulation point of $B$. Let $\{ h_d : d \in C\}$ be the definable family of functions $h_d : [0,d] \to \R$ such that $h_d$ is the compositional inverse of $t \mapsto \psi(\lambda, t )$ where $\lambda \in B$ is such that $d = |\mu - \lambda| = \psi(\lambda, 1)$.
Then $h_d$ satisfies $h_d(t) = d^{-1}t$.
It follows that $\{ h_d : d \in C \}$ is a weak pole.
\end{proof}

\begin{prop}\label{prop:uniform}
Suppose $\Cal R$ does not define a weak pole.
Then every continuous definable $f : I \to \R$ is uniformly continuous.
\end{prop}
\begin{proof}
We first treat the case when $m = 1$.
Suppose $f : I \to \R$ is continuous, definable, and not uniformly continuous.
We show that $\Cal R$ defines a weak pole.
Let $\delta > 0$ be such that for all $\epsilon > 0$ there are $t,t' \in I$ such that $|f(t) - f(t')| \geq \delta$ and $| t - t'| \leq \epsilon$.
For every $\epsilon > 0$ let
$$ A_\epsilon := \{ t \in I : |f(t) - f(t')| \geq \delta \text{   for some   } t \leq t' \leq t + \epsilon  \}.$$
Note that each $A_\epsilon$ is closed in $I$ and nonempty.
Let $p$ be a fixed element of $I$.
Let $g_0(\epsilon)$ be the maximal element of $A_\epsilon \cap (\infty,p]$ if $A_\epsilon \cap (\infty,p] \neq \emptyset$ and the minimal element of $A_\epsilon \cap [p,\infty)$ otherwise.
Note that $g_0 : \mathbb{R}_{>0} \to I$ is definable.
Let $g_1(\epsilon)$ be the least $t' \in [g_0(\epsilon),g_0(\epsilon) + \epsilon]$ such that $| f( g_0(\epsilon) ) - f(t')| \geq \delta$.
Then $g_1 : \mathbb{R}_{>0} \to I$ is definable and for all $\epsilon > 0$:
$$ 0 < g_1(\epsilon) - g_0(\epsilon)  \leq \epsilon \quad \text{and} \quad | f(g_1(\epsilon)) - f(g_0(\epsilon)) | \geq \delta.$$
We consider the definable family of functions $h_\epsilon : [0, g_1(\epsilon) - g_0(\epsilon) ] \to \mathbb{R}$ given by
$$ h_\epsilon(t) := | f( g_0(\epsilon) + t ) - f( g_0 ( \epsilon )) |.$$
Each $h_\epsilon$ is continuous. It follows from the intermediate value theorem that $[0,\delta]$ is contained in the image of every $h_\epsilon$.
Thus $\{ h_\epsilon : \epsilon \in \mathbb{R}_{>0} \}$ is a weak pole.
\end{proof}

\noindent We leave the proof of Lemma~\ref{lem:boundedabove}, an easy consequence of the triangle inequality, to the reader.
\begin{lem}\label{lem:boundedabove}
A uniformly continuous $f : I \to \R$ on a bounded open interval $I$ is bounded.
A uniformly continuous $f : \R_{>0} \to \R$ is bounded above by an affine function.
\end{lem}

\noindent Proposition~\ref{prop:uniform} and Lemma~\ref{lem:boundedabove} together yield Corollary~\ref{thm:weakpoleconseq}.

\begin{cor}\label{thm:weakpoleconseq}
Suppose $\Cal R$ does not admit a weak pole.
Then every continuous definable function on a bounded interval is bounded, and every continuous definable $f : \mathbb{R}_{>0} \to \mathbb{R}$ is bounded above by an affine function.
\end{cor}

\begin{prop}
\label{prop:weak-pole-bndd}
Suppose $\Cal R$ does not admit a weak pole.
Suppose $W$ is a bounded definable open subset of $\R^m$.
Then any continuous definable $f : W \to \R^n$ is bounded.
\end{prop}

\begin{proof}
Let $f(x) = (f_1(x),\ldots,f_n(x))$ for all $x \in W$.
It suffices to show that each $f_i : W \to \R$ is bounded.
So we suppose $n = 1$.
Given $t > 0$ we let $A_t$ be the set of $p \in W$ such that $\| p - q \| \geq t$ for all $q \in \R^m \setminus W$.
Note that each $A_t$ is closed, as $W$ is bounded it follows that each $A_t$ is compact.
Let $r > 0$ be maximal such that $A_r$ is nonempty.
Then $A_t$ is nonempty for all $0 < t < r$.
Let $g : (0,r] \to \R$ be given by
$$ g(t) = \max \{ f(p) : p \in A_t \}.$$
It is a routine analysis exercise to show that $g$ is continuous.
Corollary~\ref{thm:weakpoleconseq} shows that $g$ is bounded.
It follows that $f$ is bounded.
\end{proof}

\section{Applications}\label{section:applications}
\subsection{Extensions of Fact~\ref{fact:c2}}
We give two extensions of Fact~\ref{fact:c2}.
The first is a multivariable version of Fact~\ref{fact:c2}.

\begin{thm}
\label{thm:c2-multi}
Let $U$ be a connected definable open subset of $\R^n$, and let $f:U \to \R^m$ be definable.
\begin{enumerate}
    \item If $f$ is $C^2$ and $\Cal R$ is not field-type, then $f$ is affine.
    \item If $f$ is $C^1$ and $\Cal R$ is neither field-type nor defines an isomorphic copy of $(\Cal P(\N), \N, \in, +, \cdot)$, then $f$ is affine.
\end{enumerate}
\end{thm}

\noindent The proof is very similar to that of Lemma~\ref{lem:multivar}, so we will omit some details.

\begin{proof}
The proof of (2) follows by Theorem F and a similar argument as the proof of (1). So we only prove (1).

\noindent Suppose that $f : U \to \R^m$ is a definable $C^2$-function and $\Cal R$ is not field-type.
Let 
\[
f(x) = (f_1(x),\ldots,f_m(x)) \quad \text{for all  } x \in \R^n.
\]
It suffices to show that $f_i$ is affine for $i=1,\dots,m$.
Thus we reduce to the case that $m = 1$. \newline

\noindent As $U$ is connected, it suffices to prove that $f$ is affine on every open box contained in $U$.
Hence, without loss of generality, we may assume that $U = I_1 \times \ldots \times I_n$ is a box, where $I_1,\dots, I_n$ are open intervals.
We proceed by induction on $n$.
The base case $n = 1$ is precisely Fact~\ref{fact:c2}.
Let $U' = I_1 \times \ldots \times I_{n-1}$ and let $\pi : U \to U'$ be the projection away from the last coordinate.
For $x\in U'$, define $f_x : I_n\to \R$ by $f_x(t) = f(x,t)$ for all $t \in I_n$.
Each $f_x$ is $C^2$, so it follows that each $f_x$ is affine. \newline 

\noindent We show that $f'_x(t)$ is constant on $U'$.
Suppose not.
Following the proof of Lemma~\ref{lem:multivar} we obtain a nonempty open interval $J$ and $r,r'$ such that
the function $h : J \to \R$ given by
$$ h(t) = t^2 +t(r'-2r) -r'r' + r^2 $$
is definable.
Then $h$ is $C^2$ and non-affine, contradiction. \newline

\noindent Fix $\lambda$ such that $f'_x(t) = \lambda$ for all $(x,t) \in U' \times I_n$.
Let $g : U' \to \R$ be such that $f_x(t) = g(x) + \lambda t$ for all $(x,t) \in U$.
Since $g(x) = f_x(t) - \lambda t$ for all $(x,t) \in B$, it follows that $g$ is definable and $C^2$.
An application of induction shows that $g$ is affine. Thus $f$ is affine as well.
\end{proof}

\noindent Recall that $f : I \to \R$ is \textbf{strictly convex} if 
$$ f\left( \frac{a+b}{2} \right) < \frac{ f(a) + f(b) }{2} \quad \text{for all distinct} \quad a,b \in I. $$
A strictly convex function is continuous.

\begin{thm}
\label{thm:convex}
Suppose $\Cal R$ defines a strictly convex function.
Then $\Cal R$ is field-type.
\end{thm}

\begin{proof}
Let $f: I \to \R$ be a strictly convex definable function. Towards a contradictions, we suppose that $\Cal R$ is not of field-type. Thus $\Cal R$ is either type A or type B.
By strict convexity we know that if $x,y \in I, \epsilon > 0$ satisfy $x + \epsilon < y$ and $y + \epsilon \in I$, then 
\[
f(x + \epsilon) - f(x) < f(y + \epsilon) - f(y).
\]
Hence $f$ is not reptitious. Therefore $\Cal R$ can not be type B by Theorem E. However,
a strictly convex function is also nowhere locally affine. Thus $\Cal R$ can not be type A either by Theorem A. This is a contradiction.
\end{proof}

\subsection{An application to descriptive set theory} We give an application to descriptive set theory. We assume that the reader is familiar with basic notions of the subject (see Kechris \cite{kechris} for an introduction). Consider the Polish space $C^k([0,1])$ of all $C^k$ functions $[0,1] \to \mathbb{R}$ equipped with the topology induced by the semi-norms $f \mapsto \max_{t \in [0,1]} |f^{(j)}(t)|$ for $0 \leq j \leq k$.
Note that $C^0([0,1])$ is the space of continuous functions $[0,1] \to \mathbb{R}$ equipped with the topology of uniform convergence.
We let $C^\infty([0,1])$ be the space of smooth functions with the topology induced by the semi-norms $f \mapsto \max_{t \in [0,1]} |f^{(j)}(t)|$ for $j \in \mathbb{N}$.
Grigoriev \cite{Grigoriev} and later Le Gal \cite{LGal} constructed a comeager $Z \subseteq C^\infty([0,1])$ such that $(\mathbb{R},<,+,\cdot,f)$ is o-minimal for all $f \in Z$. The corresponding result for $C^k([0,1])$ fails.

\begin{thm}\label{thm:E}
The set of all $f \in C^k([0,1])$ such that $(\mathbb{R},<,+,f)$ is type C, is comeager in $C^k([0,1])$ for any $k \in \mathbb{N}$.
\end{thm}

\noindent While it might not be surprising that expansions of $(\R,<,+)$ by a generic bounded continuous function are not model-theoretically well behaved, Theorem \ref{thm:E} actually shows something stronger: a generic bounded continuous function defines \underline{all} bounded continuous functions over $(\mathbb{R},<,+)$. Loosely speaking, this means that given two generic functions we can recover one from the other by using finitely many boolean operations, cartesian products, and linear operations.
\begin{proof}[Sketch of proof of Theorem \ref{thm:E}]
It is well-known that the set of somewhere \linebreak $(k + 1)$-differentiable functions in $C^k([0,1])$ is meager, the case $k = 1$ being a classical result of Banach \cite{Banach}. Thus the set of all $f \in C^k([0,1])$ such that $(\mathbb{R},<,+,f)$ is type A, is meager by Theorem B. It therefore suffices to show that the collection of all $C^k$ functions $[0,1] \to \mathbb{R}$ definable in type B expansions is meager.
By Theorem E it is enough to prove that the set of reptitious $f \in C^k([0,1])$ is meager.
For each $n \geq 1$ let $A_n$ be the set of functions $f \in C^k([0,1])$ such that for some $x,y \in [0,1]$
\begin{itemize}
\item $\frac{1}{n} \leq y - x$, and
\item $ f(x + \epsilon) - f(x ) = f(y + \epsilon) - f(y) \quad \text{for all } 0 < \epsilon \leq \frac{1}{n}. $
\end{itemize}
Note that every reptitious $C^k$-function $[0,1] \to \R$ is in some $A_n$.
We show that each $A_n$ is nowhere dense. Let $n\geq 1$. As $A_n$ is a closed subset of $C^k([0,1])$, we only need to show that $A_n$ has empty interior in $C^k([0,1])$.
For every $f \in C^k([0,1])$ and $\epsilon > 0$, it is easy to construct a smooth $g: [0,1] \to \mathbb{R}$ such that $g \notin A_n$ and $| f^{(j)}(t) - g^{(j)}(t) | < \epsilon$ for all $0 \leq j \leq k$ and $t \in [0,1]$.
Thus $A_n$ has empty interior.
\end{proof}

\subsection{Applications to Automata Theory and automatic structures}\label{automata}
We finish with an application to automata theory. We first recall the terminology from \cite{BRW}. Let $r \in \N_{\geq 2}$ and $\Sigma_r = \{0,\ldots,r-1\}$. Let $x\in \R$. A \textbf{base $r$ expansion} of $x$ is an infinite $\Sigma_r\cup \{\star\}$-word $a_p\cdots a_0\star a_{-1}a_{-2}\cdots$ such that
\begin{equation}\label{eq:1}
z = -\frac{a_p}{r-1} r^p + \sum_{i=-\infty}^{p-1} a_{i} r^i
\end{equation}
with $a_p \in \{0,r-1\}$ and $a_{p-1},a_{p-2},\ldots \in \Sigma_r$. We will call the $a_i$'s the \textbf{digits} of the base $r$ expansion of $x$. The digit $a_n$ is \textbf{the digit in the position corresponding to $r^{n}$}. We define $V_r(x,u,k)$ to be the ternary predicate on $\R$ that holds whenever there exists a base $r$ expansion $a_p\cdots a_0\star a_{-1}a_{-2}\cdots$ of $x$ such that $u=r^{n}$ for some $n\in \Z$ and $a_{n} = k$. We denote by $\Cal T_r$ the expansion of $(\R,<,+)$ by $V_r$. By \cite[Lemma 3.1]{BH-Cantor} $\Cal T_r$ defines a dense $\omega$-orderable set, and by \cite[Theorem 6]{BRW} the theory of $\Cal T_r$ is decidable. Thus $\Cal T_r$ is type B and does not interpret $( \Cal P(\mathbb{N}), \mathbb{N}, \in, +, \cdot)$.\newline

\noindent The connection to automata theory arises as follows. A set $X\subseteq \R^n$ is \linebreak \textbf{$r$-recognizable} if there is a B\"uchi automaton $\Cal A$ over the alphabet $\Sigma_r^n \cup \{*\}$ which recognizes the set of all base-$r$ encodings of elements of $X$. Such B\"uchi automata are also called \textbf{real vector automata} and were introduced in Boigelot, Bronne and Rasart \cite{BBR97}. By \cite[Theorem 5]{BRW} a subset of $\R^n$ is $r$-recognizable if and only if it is $\Cal T_r$-definable without parameters. From Corollary B we immediately obtain:

\begin{cor}\label{cor:automaton}
Let $f: I \to \R$ be $C^1$ and non-affine. Then for every $r\in \N_{\geq 2}$, the graph of $f$ is not $r$-recognizable.
\end{cor}

\noindent Block Gorman et al.~\cite{CRF} prove a generalization of Corollary~\ref{cor:automaton}: if $f : I \to \R$ is differentiable and non-affine, then the graph of $f$ is not $r$-recognizable.
(This generalization was attained after the proof of Corollary~\ref{cor:automaton}, but published first.)
One advantage of the more abstract proof of Corollary~\ref{cor:automaton} is that it immediately generalizes to other enumeration systems.
The base $r$-numeration system above may be replaced by other enumeration systems such as the $\beta$-numeration system used in \cite{CLR} (when $\beta$ is a Pisot number) or the Ostrowski numeration system based on a quadratic irrational number used in \cite{H-multiplication}. These enumeration systems also give rise to type B structures with decidable theories.
Thus analogues of Corollary \ref{cor:automaton} also hold for these enumeration systems.
Results similar to Corollary~\ref{cor:automaton} have been proven, for $C^2$ functions, or for more restricted classes of automata, by Anashin~\cite{AnashinR}, Kone\v cn\'y~\cite{Konecny-auto}, and Muller~\cite{Muller-auto}. \newline

\noindent As mentioned above Abu Zaid~\cite{Zaid-thesis} has shown that $(\mathcal{P}(\N),\N,\in,+1)$ does not interpret $(\R,<,+,\cdot)$.
So $(\mathcal{P}(\N),\N,\in,+1)$ cannot interpret an expansion of $(\R,<,+)$ of field-type.
Applying this and Theorem~\ref{thm:c2-multi} we obtain the following generalization of Corollary~\ref{cor:automaton}.

\begin{cor}
Let $U$ be a connected definable open subset of $\R^n$ and let $f : U \to \R^m$ be definable and $C^1$. 
If $\Cal R$ is interpretable in $(\Cal P(\N),\N,\in,+1)$, then $f$ is affine.
In particular if $\Cal R$ is $\omega$-automatic with advice, then $f$ is affine.
\end{cor}


\bibliographystyle{abbrv}
\bibliography{HW-Bib}

\begin{thebibliography}{10}

\bibitem{Zaid-thesis}
F.~{Abu Zaid }.
\newblock {\em Algorithmic solutions via Model Theoretic Interpretations}.
\newblock PhD thesis, RWTH Aachen University, 2016.

\bibitem{ZGL}
F.~{Abu Zaid}, E.~Gr\"adel, L.~u. Kaiser, and W.~Pakusa.
\newblock Model-theoretic properties of {$\omega$}-automatic structures.
\newblock {\em Theory Comput. Syst.}, 55(4):856--880, 2014.

\bibitem{AnashinR}
V.~S. Anashin.
\newblock Quantization causes waves: smooth finitely computable functions are
  affine.
\newblock {\em p-Adic Numbers Ultrametric Anal. Appl.}, 7(3):169--227, 2015.

\bibitem{BH-Cantor}
W.~Balderrama and P.~Hieronymi.
\newblock Definability and decidability in expansions by generalized {C}antor
  sets.
\newblock {\em Preprint arXiv:1701.08426}, 2017.

\bibitem{Banach}
S.~Banach.
\newblock {\"U}ber die {B}aire'sche {K}ategorie gewisser {F}unktionenmengen.
\newblock {\em Studia Mathematica}, 3(1):174--179, 1931.

\bibitem{CRF}
A.~{B}lock Gorman, P.~Hieronymi, E.~Kaplan, R.~Meng, E.~Walsberg, Z.~Wang,
  Z.~Xiong, and H.~Yang.
\newblock Continuous regular functions.
\newblock {\em Log. Methods Comput. Sci.}, 16(1), 2020.

\bibitem{BoasWidder}
R.~P. Boas, Jr. and D.~V. Widder.
\newblock Functions with positive differences.
\newblock {\em Duke Math. J.}, 7:496--503, 1940.

\bibitem{BBR97}
B.~Boigelot, L.~Bronne, and S.~Rassart.
\newblock An improved reachability analysis method for strongly linear hybrid
  systems (extended abstract).
\newblock In {\em Computer Aided Verification}, volume 1254 of {\em Lecture
  Notes in Computer Science}, pages 167--178. Springer, 1997.

\bibitem{BRW}
B.~Boigelot, S.~Rassart, and P.~Wolper.
\newblock On the expressiveness of real and integer arithmetic automata
  (extended abstract).
\newblock In {\em Proceedings of the 25th International Colloquium on Automata,
  Languages and Programming}, ICALP '98, pages 152--163, London, UK, UK, 1998.
  Springer-Verlag.

\bibitem{Buchi}
J.~R. B{\"u}chi.
\newblock On a decision method in restricted second order arithmetic.
\newblock In {\em Logic, {M}ethodology and {P}hilosophy of {S}cience ({P}roc.
  1960 {I}nternat. {C}ongr .)}, pages 1--11. Stanford Univ. Press, Stanford,
  Calif., 1962.

\bibitem{CLR}
{\'E}.~Charlier, J.~Leroy, and M.~Rigo.
\newblock An analogue of {C}obham's theorem for graph directed iterated
  function systems.
\newblock {\em Adv. Math.}, 280:86--120, 2015.

\bibitem{Delon-powers}
F.~Delon.
\newblock {${\bf Q}$} muni de l'arithm\'{e}tique faible de {P}enzin est
  d\'{e}cidable.
\newblock {\em Proc. Amer. Math. Soc.}, 125(9):2711--2717, 1997.

\bibitem{DMS1}
A.~Dolich, C.~Miller, and C.~Steinhorn.
\newblock Structures having o-minimal open core.
\newblock {\em Trans. Amer. Math. Soc.}, 362(3):1371--1411, 2010.

\bibitem{vdd-Powers2}
L.~v.~d. Dries.
\newblock The field of reals with a predicate for the powers of two.
\newblock {\em Manuscripta Math.}, 54(1-2):187--195, 1985.

\bibitem{ElgotRabin}
C.~C. Elgot and M.~O. Rabin.
\newblock Decidability and undecidability of extensions of second (first) order
  theory of (generalized) successor.
\newblock {\em J. Symbolic Logic}, 31(2):169--181, 06 1966.

\bibitem{Engelking}
R.~Engelking.
\newblock {\em Dimension theory}.
\newblock North-Holland Publishing Co., Amsterdam-Oxford-New York; PWN---Polish
  Scientific Publishers, Warsaw, 1978.
\newblock Translated from the Polish and revised by the author, North-Holland
  Mathematical Library, 19.

\bibitem{F-Hausdorff}
A.~Fornasiero.
\newblock Expansions of the reals which do not define the natural numbers.
\newblock {\em arXiv:1104.1699, unpublished note}, 2011.

\bibitem{FHW-Compact}
A.~Fornasiero, P.~Hieronymi, and E.~Walsberg.
\newblock How to avoid a compact set.
\newblock {\em Adv. Math.}, 317:758--785, 2017.

\bibitem{FKMS}
H.~Friedman, K.~Kurdyka, C.~Miller, and P.~Speissegger.
\newblock Expansions of the real field by open sets: definability versus
  interpretability.
\newblock {\em J. Symbolic Logic}, 75(4):1311--1325, 2010.

\bibitem{FM-Sparse}
H.~Friedman and C.~Miller.
\newblock Expansions of o-minimal structures by sparse sets.
\newblock {\em Fund. Math.}, 167(1):55--64, 2001.

\bibitem{Grigoriev}
A.~Grigoriev.
\newblock On o-minimality of extensions of $\mathbb{R}$ by restricted generic
  smooth functions.
\newblock {\em arXiv:0506109}, 2005.

\bibitem{discrete}
P.~Hieronymi.
\newblock Defining the set of integers in expansions of the real field by a
  closed discrete set.
\newblock {\em Proc. Amer. Math. Soc.}, 138(6):2163--2168, 2010.

\bibitem{H-Twosubgroups}
P.~Hieronymi.
\newblock Expansions of the ordered additive group of real numbers by two
  discrete subgroups.
\newblock {\em J. Symbolic Logic}, 81(3):1007--–1027, 2016.

\bibitem{H-multiplication}
P.~Hieronymi.
\newblock When is scalar multiplication decidable?
\newblock {\em Ann. Pure Appl. Logic}, (10):1162--1175, 2019.

\bibitem{HT}
P.~Hieronymi and M.~Tychonievich.
\newblock Interpreting the projective hierarchy in expansions of the real line.
\newblock {\em Proc. Amer. Math. Soc.}, 142(9):3259--3267, 2014.

\bibitem{HW-Monadic}
P.~Hieronymi and E.~Walsberg.
\newblock Interpreting the monadic second order theory of one successor in
  expansions of the real line.
\newblock {\em Israel J. Math.}, 224(1):39--55, 2018.

\bibitem{hrushovski1993new}
E.~Hrushovski.
\newblock A new strongly minimal set.
\newblock {\em Ann. Pure Appl. Logic}, 62(2):147--166, 1993.

\bibitem{KTTT}
T.~Kawakami, K.~Takeuchi, H.~Tanaka, and A.~Tsuboi.
\newblock Locally o-minimal structures.
\newblock {\em J. Math. Soc. Japan}, 64(3):783--797, 2012.

\bibitem{kechris}
A.~S. Kechris.
\newblock {\em Classical descriptive set theory}, volume 156 of {\em Graduate
  Texts in Mathematics}.
\newblock Springer-Verlag, New York, 1995.

\bibitem{Konecny-auto}
M.~Kone\v{c}n\'y.
\newblock Real functions computable by finite automata using affine
  representations.
\newblock {\em Theoret. Comput. Sci.}, 284(2):373--396, 2002.
\newblock Computability and complexity in analysis (Castle Dagstuhl, 1999).

\bibitem{LasStein}
M.~C. Laskowski and C.~Steinhorn.
\newblock On o-minimal expansions of {A}rchimedean ordered groups.
\newblock {\em J. Symbolic Logic}, 60(3):817--831, 1995.

\bibitem{LGal}
O.~Le~Gal.
\newblock A generic condition implying o-minimality for restricted
  {$C^\infty$}-functions.
\newblock {\em Ann. Fac. Sci. Toulouse Math. (6)}, 19(3-4):479--492, 2010.

\bibitem{PL}
J.~Loveys and Y.~Peterzil.
\newblock Linear o-minimal structures.
\newblock {\em Israel J. Math.}, 81(1-2):1--30, 1993.

\bibitem{MPP}
D.~Marker, Y.~Peterzil, and A.~Pillay.
\newblock Additive reducts of real closed fields.
\newblock {\em J. Symbolic Logic}, 57(1):109--117, 1992.

\bibitem{opencore}
C.~Miller and P.~Speissegger.
\newblock Expansions of the real line by open sets: o-minimality and open
  cores.
\newblock {\em Fund. Math.}, 162(3):193--208, 1999.

\bibitem{Muller-auto}
J.-M. Muller.
\newblock Some characterizations of functions computable in on-line arithmetic.
\newblock {\em IEEE Trans. Comput.}, 43(6):752--755, 1994.

\bibitem{PS-Tri}
Y.~Peterzil and S.~Starchenko.
\newblock A trichotomy theorem for o-minimal structures.
\newblock {\em Proc. London Math. Soc. (3)}, 77(3):481--523, 1998.

\bibitem{PSS}
A.~Pillay, P.~Scowcroft, and C.~Steinhorn.
\newblock Between groups and rings.
\newblock {\em Rocky Mountain J. Math.}, 19(3):871--885, 1989.
\newblock Quadratic forms and real algebraic geometry (Corvallis, OR, 1986).

\bibitem{RSW}
J.-P. Rolin, P.~Speissegger, and A.~J. Wilkie.
\newblock Quasianalytic {D}enjoy-{C}arleman classes and o-minimality.
\newblock {\em J. Amer. Math. Soc.}, 16(4):751--777, 2003.

\bibitem{Simon-Book}
P.~Simon.
\newblock {\em A guide to NIP theories}, volume~44 of {\em Lecture Notes in
  Logic}.
\newblock Cambridge University Press, 2015.

\bibitem{Zilberconj2}
B.~Zil'ber.
\newblock Strongly minimal countably categorical theories. ii.
\newblock {\em Siberian Mathematical Journal}, 25(3):396--412, 1984.

\end{thebibliography}
\end{document}